\documentclass[12pt]{amsart}
\usepackage{}
\usepackage{amssymb}
\usepackage{amsfonts,latexsym,amsmath, amssymb}
\usepackage{mathrsfs,MnSymbol}
\usepackage{url,color}
\usepackage{upgreek}
\usepackage{fancyhdr}
\usepackage{hyperref}
\usepackage{times}
\usepackage{scalefnt}
\usepackage{url}
\usepackage{cite}
\usepackage{amsmath}
\usepackage{slashed}

\def\({\left(}
\def\){\right)}
\def\p{\partial}
\def\.{\cdot}
\def\R{\mathbb{R}}
\def\s{\,\,\,\,\,\,}

\def\e{\varepsilon}
\def\m{\mathcal{M}}
\def\D{\slashed{D}}
\def\A{\mathcal{A}}
\def\L{\mathcal{L}}
\def\N{\mathcal{N}}
\def\B{\mathcal{B}}
\def\W{\mathcal{W}}
\def\ex{\overline{\exp}}
\newtheorem{theorem}{Theorem}[section]
\newtheorem{lemma}[theorem]{Lemma}

\newtheorem{corollary}[theorem]{Corollary}
\newtheorem{definition}[theorem]{Definition}
\newtheorem{remark}[theorem]{Remark}

\numberwithin{equation}{section}

\topmargin - .23 in

\begin{document}

\title{Neumann Boundary Problem of Second Order Parabolic Quasi-Linear System with  Variable Coefficient on a Vector Bundle}
\author{Zonglin Jia}

\begin{abstract}
Classical results of second order parabolic quasi-linear equations always require that the nonlinear terms are controlled by a power of the unknown functions and their first derivatives. We improve the previous results. More precisely, in the present article the upper bound of nonlinear term can be extended to a power series.
\end{abstract}
\maketitle

\setcounter{tocdepth}{2}
\pagenumbering{roman} \tableofcontents \newpage \pagenumbering{arabic}
\section{Introduction}
As is well known, second order parabolic system is very important in math and physics. Ones of the most studied are the equations of the following forms
\begin{eqnarray*}
\left\{
\begin{array}{rl}
&\p_tu(t,x)=\Delta u(t,x)+F(u(t,x),\p u(t,x)),\s\s(t,x)\in\R^+\times\Omega\\
&\p_{\vec{n}}u=0,\s\s(t,x)\in\R^+\times\p\Omega\\
&u(0,x)=u_0(x)
\end{array}
\right.
\end{eqnarray*}
with $\Omega$ being a bounded domain in Euclidean Space, or
\begin{eqnarray*}
\left\{
\begin{array}{rl}
&\p_tu(t,x)=\Delta u(t,x)+F(u(t,x),\p u(t,x)),\s\s(t,x)\in\R^+\times M\\
&u(0,x)=u_0(x)
\end{array}
\right.
\end{eqnarray*}
with $M$ being a closed Riemannian manifold. To the best of our knowledge, all of their nonlinear terms are required to satisfy power conditons, i.e. 
$$
|F(U,P)|\lesssim(|U|+|P|)^{\alpha}\s\s\mbox{and}\s\s|\p_UF(U,P)|+|\p_PF(U,P)|\lesssim(|U|+|P|)^{\alpha-1}.
$$
The literatures in this area are refered to \cite{C}, \cite{T} and \cite{W}. 

In this article, we study Neumann boundary problem of second order parabolic quasi-Linear system with  variable coefficients on a vector bundle and our nonlinear term is assumed to meet \textbf{series condition}, i.e. the nonlinear term is bounded by a power series. These details are the following. 

$(M,g)$ with the Riemannian metric $g$ is a compact $m$-dimensional Manifold. $x:=(x^1,\cdots,x^m)$ stands for a point in $M$. Define $J(m):=m$ if $m\geq2$ and $J(1)=2$. For two tensors $S$ and $T$, $S\sharp T$ means a new tensor, which  constractes $S\otimes T$ by $g^{-1}$.

Throughout this article, the same indices appearing twice means summing it. To state our main results, let us give some definitions.

\begin{definition}
The spatial tangent bundle $\mathcal{S}M$ is a vector bundle on $[0,T)\times M$, which is a subset of the tangent bundle $T\([0,T)\times M\)$ satisfing
$$
\mathcal{S}M:=\{S\in T\([0,T)\times M\)|dt(S)=0\}.
$$
\end{definition}
Let $(E,\langle\.,\.\rangle,\D)$ be a vector bundle over $[0,T)\times M$ satisfing
\begin{eqnarray*}
d\langle\omega_1,\omega_2\rangle=\langle\D\omega_1,\omega_2\rangle+\langle\omega_1,\D\omega_2\rangle\s\s\mbox{and}\s\s \D_t\D_X\omega=\D_X\D_t\omega
\end{eqnarray*}
for all $X\in \mathcal{S}M$.

Furthermore, we define a vector bundle $\p E:=E|_{[0,T]\times\p M}$ over $[0,T]\times\p M$ and we restrict the metric $\langle\.,\.\rangle$ and the connection $\D$ on $\p E$ , which are denoted by the same symbols.

\begin{definition}
$A\in\mbox{Hom}(E,E)$ are called bounded if $|A(\omega)|_{E}\lesssim|\omega|_{E}$.
\end{definition}

\begin{definition}
$A\in\mbox{Hom}(E,E)$ is called positive-definite if $\langle A(\omega_1),\omega_2\rangle=\langle \omega_1,A(\omega_2)\rangle$ and $|A(\omega)|_{E}\gtrsim|\omega|_{E}$.
\end{definition}

Theorem \ref{theorem2} is the foundation to quasi-linear case.

\begin{theorem}\label{theorem2}
For the next system 
\begin{eqnarray}\label{5}
\left\{
\begin{array}{rl}
&\D_t\omega-\D_x[A(\D_x\omega)]\sharp g^{-1}+B(\D_Y\omega)+C(\omega)=F\\
&\omega(0,\.)=\omega_0,\s\s\D_{\vec{n}}\omega=\theta
\end{array}
\right.
\end{eqnarray}
we assume that $A,B,C$ are in $\mbox{Hom}(E,E)$ and $\vec{n}$ the inner normal vector. $F$ is a known section on $E$ and $Y\in\mathcal{S}M$ is a known vector field. Furthermore, $A, B, C, F, Y$, $\omega_0$ are smooth and $\theta$ is $C^1$-differentiable with respect to $x$. Then $(\ref{5})$ admits a unique smooth solution,  provided $A$ is positive-definite.
\end{theorem}

To describe Theorem \ref{theorem1.2}, we firstly make some preparation.
\subsection{Some function spaces}
We define the Sobolev space $W^{l,r}_p$ as
$$
W^{l,r}_p:=\{\omega\in L^{\infty}\([0,T),L^p(M)\)|\s||\omega||_{W^{l,r}_p}<\infty\}
$$
with
$$
||\omega||^p_{W^{l,r}_p}:=\sup\limits_{t\in[0,T)}\left\{\sum\limits_{k=0}^l||\D_x^k\omega(t)||^p_{L^p(M)}+\sum\limits_{k=0}^r||\D^k_t\omega(t)||^p_{L^p(M)}\right\}.
$$
Besides, we define a norm on $C^{\infty}(E)$ as
$$
||\omega||^2:=||\omega||^2_{L^2(E)}+\int_{\R}\int_M|D_x\omega|^2+\int_M|\omega(T,\cdot)|^2+\int_M|\omega(0,\cdot)|^2+\int_0^T\int_{\p M}|\omega|^2.
$$

Before  illustrating Theorem \ref{theorem1.2}, we shall give some definitions of constants. $c(p)$ is the Sobolev's Constant such that
$$
||\omega||_{L^p}\leq c(p)||\omega||_{W^{\lceil p\rceil_J,0}_J}
$$
and $C(k,p)$ is the constant in Lemma $\ref{lemma2.1}$. $	C$ denotes a universal constant appearing in Section \ref{section3}. $C(||A||_{W_{\infty}^{n,0}})$ means a constant depended upon $||A||_{W_{\infty}^{n,0}}$.
\begin{theorem}\label{theorem1.2}
Assume that $A$, who is positive-definite and smooth, satisfies $\D_tA=0$ and \textbf{Condition 2:} 
$g^{-1}\sharp\langle A(\D_x\omega),\omega\rangle\otimes\langle\D_x\omega,\omega\rangle\geq0$ for any $\omega\in E$.
Define $\L_A\omega:=-g^{-1}\sharp\D_x[A(\D_x\omega)]$. Moreover, suppose that $F$ is smooth and satisfies the \textbf{Series Condition:}
$$
|\D_t^p\D_x^sD^q_{\eta}D^r_{\rho}F(t,x,\eta,\rho)|\leq C_{psqr}\(|\eta|,|\rho|\):=\sum\limits_{u+v=1}C_{uv}^{psqr}|\eta|^u|\rho|^v
$$
for some smooth function $C_{psqr}$ which is non-decreasing with respect to the two variables. Furthermore, we impose \textbf{Condition 1:}
$$
c((w+u)b(n))^{2w+2u}\left(\sum\limits_{a+b+q+r=0}^{n-2}|C^{abqr}_{wu}|\right)\leq n^{w+u}/[L(w+u)!]
$$
with
$$b(n):=(3+n/2+3n^2/2)J$$
and $L:=\sup\limits_{0\leq t<T}\exp\left[J2^9C(2,J)^4t^{2/J}e^{2\tilde{C}}+8\mathfrak{L}(t)^2\right]
$, where 
\begin{eqnarray*}
&&\mathfrak{L}(t):=||\omega_0||_{W^{2,0}_J}\\
&&+e^{Ct}\int_0^tds\,\left\{||\L_A[F(s,\.,\omega_0,\D_x\omega_0)]-\L_A^2\omega_0||^J_{L^J}+||\D_{\vec{n}}[F(s,\.,\omega_0,\D_x\omega_0)]-\D_{\vec{n}}\L_A\omega_0||^J_{L^J(\p M)}\right\},
\end{eqnarray*} 
\begin{eqnarray*}
\tilde{C}:=\sup\limits_{n\in\mathbb{N}}18^n(n!)^3\sup\limits_{t\in[0,T)}\left[\tau(n,n,t)+||\theta(t)||^J	_{W^{1,0}_J}+\mathfrak{B}^{1+n}(t)\right]<\infty,
\end{eqnarray*}
\begin{eqnarray*}
\tau(n,n,t):=\log\left\{3+2||\theta(t)||^J_{W^{1,0}_J}+||\D_t^n\theta(t)||^J_{L^J}\right\}+\log\left[\Psi(n,n)(1+t)\right]+2Ct+\mu(0,n)n!,
\end{eqnarray*}
$$
\Psi(n,n):=2^{27Jn^3}n^{8Jn^2}n^{3Jn^3}C(2,J)^{4Jn^2}c((1+n)b(1+n))^{8J(1+n)^3}J^{26Jn^3},
$$
\begin{eqnarray*}
&&\mu(0,n):=(J+2)\sum\limits_{j=0}^{n-1}\frac{\log(n-j)}{(n-j)!}+J\sum\limits_{j=0}^{n-1}\frac{\log C(||A||_{W^{(2n-3)j+2n,0}_{\infty}})}{(n-j)!}\\
&&+J\log2\sum\limits_{j=0}^{n-1}\frac{(2n-4)j+3n-3}{(n-j)!}+\sum\limits_{j=0}^{n-1}\frac{\log\left(S([2n-4]j+3n-3)\right)}{(n-j)!}\nonumber\\
&&+\sum\limits_{j=0}^{n-1}\frac{\log[(2n-3)j+2n]}{(n-j)!}+J\log2\sum\limits_{j=0}^{n-1}\frac{(2n-4)j+3n-3}{(n-j)!}\\
&&+\sum\limits_{j=0}^{n-1}J[(2n-4)j+3n-3]\frac{\log\{c(J[(2n-4)j+3n-3]^2)\}}{(n-j)!}\\
&&+\log\left[1+\lambda((2n-1)(n-1))\right],
\end{eqnarray*}
$$
S(n):=\sum\limits_{a+b+f+u=0}^{n}||C_{abfu}(|\omega_0|,|\D_x\omega_0|)||_{L^{\infty}}^J,
$$
$$
\lambda(n):=||\omega_0||^J_{W^{n,0}_J}+C(||A||_{W_{\infty}^{n+2,0}})^J||F(0,\.,\omega_0,\D_x\omega_0)||^J_{W^{n,0}_J}
$$
and
\begin{eqnarray*}
&&\mathfrak{B}^n(t):=||\omega_0||^J_{L^J}/J+Ce^{Ct}\left[\sum\limits_{r+q=0}^{n-1}||\L_A^q[\D_t^rF(0,\.,\omega_0,\D_x\omega_0)]||_{L^J}^J+\sum\limits_{i=0}^n||\L_A^i\omega_0||_{L^J}^J\right]\\
&&+Ce^{Ct}\int_0^t\left[\sum\limits_{i=0}^n||\D_t^iF(s,\.,\omega_0,\D_x\omega_0)||^J_{L^J}+\sum\limits_{i=0}^n||\D_t^{i}\theta(s,\.)||_{L^J(\p M)}^J\right]ds.
\end{eqnarray*}
Then, provided that $\theta$ is $C^1$-differentiable with respect to $x$ and $C^{\infty}$-differentiable with respect to $t$ and $\omega_0$ is smooth, the following equation admits a unique  solution $\omega\in \bigcap\limits_{r+2l\leq 2k+2}L^{\infty}([0,T),W^{r,l}_J)$ for small $T>0$
\begin{eqnarray}\label{eq:2}
\left\{
\begin{array}{rl}
&\D_t\omega+\L_A\omega=F(t,x,\omega,\D_x\omega)\\
&\omega(0,\.)=\omega_0,\s\D_{\vec{n}}\omega=\theta,
\end{array}
\right.
\end{eqnarray}
where $k$ is the largest integer such that \begin{eqnarray*}
||\theta(t)||^J	_{W^{1,0}_J}<\exp\{[(k-1)!]^{-3}\tilde{C}\}-1\s\s\mbox{for $t\in[0,T)$}.
\end{eqnarray*} 
\end{theorem}
\begin{remark}
If $\theta=0$, then $\omega$ is $C^{\infty}$-smooth. 
\end{remark}
\begin{remark}
Theorem \ref{theorem1.2} essentially tells us that, since of the high non-linearity of nonlinear term, the second order parabolic quasi-linear system admits a unique local $C^k$-solution with small initial and boundary datum, even though the initial data is $C^{\infty}$-smooth. Moreover, $k$ depends on the size of boundary data and therefore is $\epsilon$-regularity.
\end{remark}
\section{Proof of the linear system}\label{section2}
\begin{definition}
We call that $\omega\in L^2(E)$ with $||\omega||<\infty$ is a weak solution to (\ref{5}) if
\begin{eqnarray*}
&&\int_M\langle\omega(T,\cdot),\varphi(T,\cdot)\rangle+\int_0^T\int_M\langle A(\D_x\omega),\D_x\varphi\rangle\sharp g^{-1}+\int_0^T\int_M\langle B(\D_Y\omega),\varphi\rangle\\
&&+\int_0^T\int_M\langle C(\omega),\varphi\rangle-\int_0^T\int_M\langle\omega,\D_t\varphi\rangle=\int_0^T\int_M\langle F,\varphi\rangle+\int_M\langle\omega_0,\varphi(0,\cdot)\rangle-\int_0^T\int_{\p M}\langle A(\theta),\varphi\rangle
\end{eqnarray*}
for all $\varphi\in C^{\infty}(E)$.
\end{definition}
\begin{theorem}
$\omega\in W^{1,1}_2$ with $||\omega||<\infty$ is a weak solution to (\ref{5}) if and only if
\begin{eqnarray*}
&&\mathfrak{B}_{\lambda}(\omega,\varphi):=\int_0^T ds\int_Me^{-\lambda s}\langle A(\D_x\omega),\D_x\D_t\varphi\rangle\sharp g^{-1}+\int_0^T ds\int_Me^{-\lambda s}\langle C(\omega),\D_t\varphi\rangle\\
&&+\int_0^T ds\int_Me^{-\lambda s}\langle\D_t\omega,\D_t\varphi\rangle+\int_0^T ds\int_Me^{-\lambda s}\langle B(\D_Y\omega),\D_t\varphi\rangle+\int_M\langle\omega(0,\.),\D_t\varphi(0,\.)\rangle\\
&=&\int_0^T ds\int_Me^{-\lambda s}\langle F,\D_t\varphi\rangle-\int_0^T\int_{\p M}e^{-\lambda s}\langle A(\theta),\D_t\varphi\rangle+\int_M\langle\omega_0,\D_t\varphi(0,\.)\rangle:=\mathfrak{F}(\varphi)
\end{eqnarray*}
for some $\lambda\in\R$ and any $\varphi\in C^{\infty}(E)$.
\end{theorem}
\begin{proof}
Apply the same trick of Proposition 3.6 in \cite{B}.
\end{proof}

\begin{theorem}
There exists a unique function $\omega\in W^{1,1}_2$ with $||\omega||<\infty$ such that $\mathfrak{B}_{\lambda}(\omega,\varphi)=\mathfrak{F}(\varphi)$ for any large positive $\lambda$ and any $\varphi\in C^{\infty}(E)$, provided that $\D_tA, A, B, C, Y, F$ are bounded and $A$ is positive-definite.
\end{theorem}

\begin{proof}
The approach is similar to that of Theorem 3.8 in \cite{B}.
\end{proof}

\begin{corollary}
There exists a unique weak solution $\omega\in W^{1,1}_2$ with $||\omega||<\infty$ to $(\ref{5})$, provided $\D_tA, A, B, C, Y, F$ are bounded and $A$ is positive-definite.
\end{corollary}

\begin{theorem}
Suppose $\omega\in W^{1,1}_2$ with $||\omega||<\infty$ is a weak solution to $(\ref{5})$. If $A, B, C$, $F, Y, \omega_0$ are smooth and $\theta$ is $C^1$-differentiable with respect to $x$, then $\omega$ is also smooth, provided $A$ is positive-definite. Moreover, the smooth solution is unique.
\end{theorem}

\begin{proof}
Assume that $\{(U_i,\Phi_i), (V_j,\Psi_j)|U_i\cap\p M=\emptyset\}$ is a finite coverage of $M$ and $\{(V_j\cap\p M,\Psi_j)\}$ is a finite coverage of $\p M$ such that $\Phi_i(U_i)=\mathbb{S}(0,1)$ and $\Psi_j(V_j)=\mathbb{S}^+(0,1)$, where $\mathbb{S}(0,r):=\{x\in\R^m|\max\{|x^i|,1\leq i\leq m\}<r\}$ and $\mathbb{S}^+(0,r):=\{x\in\R^m|x^m\geq0,\max\{|x^i|,1\leq i\leq m\}<r\}$. It is not difficult to get
$$\Psi_j(V_j\cap\p M)=\{x=(x^1,x^2,\cdots,x^{m-1},0)|\max\{|x^i|,1\leq i\leq m-1\}<1\}.$$
Without loss of generality, we suppose that $E|_{[0,T]\times U_i}$ and $E|_{[0,T]\times V_j}$ are respectively isomorphic to $[0,T]\times\mathbb{S}(0,1)\times\R^L$ and $[0,T]\times\mathbb{S}^+(0,1)\times\R^L$.

Taking any $\varphi\in C^{\infty}([0,T], C_c^{\infty}(U_i))$ as test function, we have
\begin{eqnarray*}
&&\int_0^T ds\int_M\langle A(\D_x\omega),\D_x\varphi\rangle\sharp g^{-1}+\int_0^T ds\int_M\langle C(\omega),\varphi\rangle+\int_0^T ds\int_M\langle\D_t\omega,\varphi\rangle\\
&&+\int_0^T ds\int_M\langle B(\D_Y\omega),\varphi\rangle+\int_M\langle\omega(0,\.),\varphi(0,\.)\rangle=\int_0^T ds\int_M\langle F,\varphi\rangle+\int_M\langle\omega_0,\varphi(0,\.)\rangle.
\end{eqnarray*}
From the regularity theory of parabolic equation it follows that $\omega\in C^{\infty}([0,T]\times U_i)$.

Now we focus on the case of boundary. Take any $\varphi\in C^{\infty}([0,T]\times V_j)$ with $\mbox{supp}[\varphi(t,\,)]\subset V_j$. For simplicity $\omega\circ\Psi_j^{-1}$ and $\varphi\circ\Psi_j^{-1}$ are also denoted by $\omega$ and $\varphi$. Hence, it is not hard to get
\begin{eqnarray}\label{eq}
&&\int_0^T ds\int_{\mathbb{S}^+(0,1)}\langle A(\p_i\omega),\p_i\varphi\rangle+\int_0^T ds\int_{\mathbb{S}^+(0,1)}\langle C(\omega),\varphi\rangle+\int_0^T ds\int_{\mathbb{S}^+(0,1)}\langle\p_t\omega,\varphi\rangle\\
&&+\int_0^T ds\int_{\mathbb{S}^+(0,1)}\langle B(\p_Y\omega),\varphi\rangle+\int_{\mathbb{S}^+(0,1)}\langle\omega(0,\.),\varphi(0,\.)\rangle=\int_0^T ds\int_{\mathbb{S}^+(0,1)}\langle F,\varphi\rangle\nonumber\\
&&+\int_{\mathbb{S}^+(0,1)}\langle\omega_0,\varphi(0,\.)\rangle-\int_0^T\int_{\substack{x\in\mathbb{S}^+(0,1),\\x^m=0}}\langle A(\theta),\varphi\rangle.\nonumber
\end{eqnarray}
Let $e_k$ be the unit vector whose the $k$-th element is 1 and the others are all 0. Define the shift operator $(\tau_h^k\varphi)(t,x)=\varphi(t,x+he_k)$ and the difference operator $D_h^k\varphi:=(\tau_h^k\varphi-\varphi)/h$ for $h\not=0$.

For $1\leq k\leq m-1$ and $h\in(0,1/16)$, take $D_{-h}^k(\zeta^2 D_{h}^k\omega)$ as test function, where $\zeta\in C^{\infty}(\mathbb{S}^+(0,3/4))$, $0\leq\zeta\leq 1$ and $|\p\zeta|+|\p^2\zeta|\lesssim 1$,
$$\zeta(x)=0\s\mbox{for}\s\max\{|x^i|\,\,|1\leq i\leq m\}>3/4;\s\s\zeta(x)=1\s\mbox{for}\s\max\{|x^i|\,\,|1\leq i\leq m\}<1/2.$$
We focus on the key terms
\begin{eqnarray*}
&&-\int_0^T ds\int_{\mathbb{S}^+(0,1)}\langle A(\p_i\omega),\p_iD_{-h}^k(\zeta^2 D_{h}^k\omega)\rangle\\
&=&\underbrace{\int_0^T ds\int_{\mathbb{S}^+(0,1)}\langle (D_h^kA)(\p_i\tau_h^k\omega),2\zeta\p_i\zeta\.D_h^k\omega\rangle}_{I_1}+\underbrace{\int_0^T ds\int_{\mathbb{S}^+(0,1)}\langle (D_h^kA)(\p_i\tau_h^k\omega),\zeta^2\.D_h^k\p_i\omega\rangle}_{I_2}\\
&&+\underbrace{\int_0^T ds\int_{\mathbb{S}^+(0,1)}\langle A(\p_iD_h^k\omega),2\zeta\p_i\zeta\.D_h^k\omega\rangle}_{I_3}+\underbrace{\int_0^T ds\int_{\mathbb{S}^+(0,1)}\langle A(\p_iD_h^k\omega),\zeta^2\.D_h^k\p_i\omega\rangle}_{I_4}.
\end{eqnarray*}
Recalling that $A$ is positive-definite and smooth, we have
$$
I_{4}\gtrsim\int_0^T ds\int_{\mathbb{S}^+(0,1)}\zeta^2|D^k_h\p\omega|^2,\s\s I_{3}\lesssim\int_0^T ds\int_{\mathbb{S}^+(0,1)}|\p D^k_h\omega|\.|D^k_h\omega|\.|\p\zeta|\.|\zeta|.
$$
and
$$
I_1\lesssim\int_0^T ds\int_{\mathbb{S}^+(0,1)}|\p\tau^k_h\omega|\.|D^k_h\omega|\.|\p\zeta|\.|\zeta|,
\s\s
I_2\lesssim\int_0^T ds\int_{\mathbb{S}^+(0,1)}\zeta^2\.|\p\tau^k_h\omega|\.|\p D^k_h\omega|.
$$
Simple computation leads to
$$
2\int_0^T ds\int_{\mathbb{S}^+(0,1)}\langle\p_t\omega,D^k_{-h}(\zeta^2 D^k_h\omega)\rangle=-\int_{\mathbb{S}^+(0,1)}\zeta^2|D^k_h\omega(0,\.)|^2+\int_{\mathbb{S}^+(0,1)}\zeta^2|D^k_h\omega(T,\.)|^2
$$
and
\begin{eqnarray*}
\int_0^T\int_{\substack{x\in\mathbb{S}^+(0,1),\\x^m=0}}\langle A(\theta),D_{-h}^k(\zeta^2 D_h^k\omega)\rangle&=&\int_0^T\int_{\substack{x\in\mathbb{S}^+(0,1),\\x^m=0}}\langle (D_h^kA)(\tau_h^k\theta),D_h^k\omega\rangle\zeta^2\\
&&-\int_0^T\int_{\substack{x\in\mathbb{S}^+(0,1),\\x^m=0}}\langle A(D_h^k\theta),D_h^k\omega\rangle\zeta^2,
\end{eqnarray*}
which implies
\begin{eqnarray*}
&&\Big|\int_0^T\int_{\substack{x\in\mathbb{S}^+(0,1),\\x^m=0}}\langle A(\theta),D_{-h}^k(\zeta^2 D_h^k\omega)\rangle\Big|\lesssim\int_0^T\int_{\substack{x\in\mathbb{S}^+(0,1),\\x^m=0}}(|\theta|+|D^k_h\theta|)|D_h^k\omega|.
\end{eqnarray*}
Classical trick gives
\begin{eqnarray*}
&&\sum\limits_{k=1}^{m-1}(\int_0^Tds\int_{\mathbb{S}^+(0,1)}\zeta^2|D_h^k\p\omega|^2-\int_{\mathbb{S}^+(0,1)}\zeta^2|D_h^k\omega(0,\.)|^2+\int_{\mathbb{S}^+(0,1)}\zeta^2|D_h^k\omega(T,\.)|^2)\\
&\lesssim&\sum\limits_{k=1}^{m-1}\int_0^Tds\int_{\mathbb{S}^+(0,1)}|\p\zeta|^2|D_h^k\omega|^2+\int_0^Tds\int_{\mathbb{S}^+(0,1)}|\p\omega|^2-\sum\limits_{k=1}^{m-1}\int_{\mathbb{S}^+(0,1)}\zeta^2|D_h^k\omega_0|^2\\
&&+\sum\limits_{k=1}^{m-1}\int_0^Tds\int_{\mathbb{S}^+(0,1)}|D_h^kF|^2+\sum\limits_{k=1}^{m-1}\int_0^T\int_{\substack{x\in\mathbb{S}^+(0,1),\\x^m=0}}|D_h^k\omega|^2+\int_0^T\int_{\substack{x\in\mathbb{S}^+(0,1),\\x^m=0}}(|\theta|^2+\sum\limits_{k=1}^{m-1}|D^k_h\theta|^2).
\end{eqnarray*}
Letting $h$ tend to $0$, we have
\begin{eqnarray*}
&&\sum\limits_{k=1}^{m-1}\(\int_0^Tds\int_{\mathbb{S}^+(0,1)}\zeta^2|\p_k\p\omega|^2-\int_{\mathbb{S}^+(0,1)}\zeta^2|\p_k\omega(0,\.)|^2+\int_{\mathbb{S}^+(0,1)}\zeta^2|\p_k\omega(T,\.)|^2\)\\
&\lesssim&\sum\limits_{k=1}^{m-1}\int_0^Tds\int_{\mathbb{S}^+(0,1)}|\p\zeta|^2|\p_k\omega|^2+\int_0^Tds\int_{\mathbb{S}^+(0,1)}|\p\omega|^2-\sum\limits_{k=1}^{m-1}\int_{\mathbb{S}^+(0,1)}\zeta^2|\p_k\omega_0|^2\\
&&+\sum\limits_{k=1}^{m-1}\int_0^Tds\int_{\mathbb{S}^+(0,1)}|\p_kF|^2+\sum\limits_{k=1}^{m-1}\int_0^T\int_{\substack{x\in\mathbb{S}^+(0,1),\\x^m=0}}|\p_k\omega|^2+\int_0^T\int_{\substack{x\in\mathbb{S}^+(0,1),\\x^m=0}}\(|\theta|^2+\sum\limits_{k=1}^{m-1}|\p_k\theta|^2\),
\end{eqnarray*}
which implies that $\p_k\p_x\omega$ exists for $1\leq k\leq m-1$.

As for the case when $k=m$, we extend $\omega$ as
\begin{eqnarray*}
\tilde{\omega}(t, x^1,\cdots,x^{m-1},x^m)&:=&\omega(t, x^1,\cdots,x^{m-1},0)+x^m\theta(t, x^1,\cdots,x^{m-1})\\
&:=&\(\mathcal{T}^m\omega\)(t,x^1,\cdots,x^{m-1},x^m)+x^m\theta(t, x^1,\cdots,x^{m-1})
\end{eqnarray*}
for $x^m<0$. Take $D_{h}^m(\xi^2 D_{-h}^m\tilde{\omega})$ as test function, where $\xi\in C^{\infty}(\mathbb{S}^+(0,1))$, $\xi(x^1,\cdots,x^{m-1},x^m)=1$ for $0\leq x^m<1/2$ and $\xi(x^1,\cdots,x^{m-1},x^m)=0$ for $3/4\leq x^m<1$.

We focus on the key terms
\begin{eqnarray*}
&&-\int_0^T ds\int_{\mathbb{S}^+(0,1)}\langle A(\p_i\omega),\p_iD_{h}^m(\xi^2 D_{-h}^m\tilde{\omega})\rangle\\
&=&\underbrace{\frac{1}{h}\int_0^T\int_{\mathbb{S}^+(0,1)}\langle A(\p_i\omega),2\xi\p_i\xi D^m_{-h}\tilde{\omega}+\xi^2 D^m_{-h}\p_i\tilde{\omega}\rangle}_{I_7}\\
&&-\underbrace{\frac{1}{h}\int_0^T\int_{\mathbb{S}^+(0,1)}\langle (\tau^m_hA)(\tau^m_h\p_i\omega),2\tau^m_h\xi\.\tau^m_h\p_i\xi\.D^m_h\omega+(\tau^m_h\xi)^2\. D^m_{h}\p_i\omega\rangle}_{I_8}\\
&&+\underbrace{\int_0^T\int_{\mathbb{S}^+(0,1)}\langle D^m_h[A(\p_i\omega)],2\tau^m_h\xi\.\tau^m_h\p_i\xi\.D^m_h\omega+(\tau^m_h\xi)^2\.D^m_h\p_i\omega\rangle}_{I_9}.
\end{eqnarray*}
Simple computation gives
\begin{eqnarray*}
I_7&=&I_8+\frac{1}{h^2}\sum\limits_{k=1}^{m-1}\int_0^T\int_{0\leq x^m<h}\langle A(\p_k\omega),(\mathcal{T}^m\p_k\p_m\omega-\p_k\theta)x^m+h\p_k\theta+o(x^m)\rangle\\
&&+\frac{1}{h^2}\int_0^T\int_{0\leq x^m<h}\langle A(\p_m\omega),\p_m\omega-\theta\rangle
\end{eqnarray*}
and
\begin{eqnarray*}
I_9\gtrsim\int_0^T\int_{\mathbb{S}^+(0,1)}|\tau^m_h\xi|^2|D^m_h\p_x\omega|^2-\int_0^T\int_{\mathbb{S}^+(0,1)}|\p_x\omega|^2-\int_0^T\int_{\mathbb{S}^+(0,1)}|D^m_h\omega|^2.
\end{eqnarray*}
Moreover,
$$
\int_0^T\int_{\substack{x\in\mathbb{S}^+(0,1),\\x^m=0}}\langle A(\theta),D_{h}^m(\xi^2 D_{-h}^m\tilde{\omega})\rangle=\frac{1}{h^2}\int_0^T\int_{0\leq x^m<h}\langle A(\theta),\p_m\omega-\theta\rangle
$$
implies
\begin{eqnarray*}
&&-\int_0^T ds\int_{\mathbb{S}^+(0,1)}\langle A(\p_i\omega),\p_iD_{h}^m(\xi^2 D_{-h}^m\tilde{\omega})\rangle-\int_0^T\int_{\substack{x\in\mathbb{S}^+(0,1),\\x^m=0}}\langle A(\theta),D_{h}^m(\xi^2 D_{-h}^m\tilde{\omega})\rangle\\
&\gtrsim&\frac{1}{h^2}\sum\limits_{k=1}^{m-1}\int_0^T\int_{0\leq x^m<h}\langle A(\p_k\omega),(\mathcal{T}^m\p_k\p_m\omega-\p_k\theta)x^m+h\p_k\theta+o(x^m)\rangle\\
&&+\frac{1}{h^2}\int_0^T\int_{0\leq x^m<h}\langle A(\p_m\omega-\theta),\p_m\omega-\theta\rangle-\int_0^T\int_{\mathbb{S}^+(0,1)}|\p_x\omega|^2-\int_0^T\int_{\mathbb{S}^+(0,1)}|D^m_h\omega|^2\\
&&+\int_0^T\int_{\mathbb{S}^+(0,1)}|\tau^m_h\xi|^2|D^m_h\p_x\omega|^2\\
&\gtrsim&\int_0^T\int_{\mathbb{S}^+(0,1)}|\tau^m_h\xi|^2|D^m_h\p_x\omega|^2-\int_0^T\int_{\mathbb{S}^+(0,1)}|\p_x\omega|^2-\int_0^T\int_{\mathbb{S}^+(0,1)}|D^m_h\omega|^2\\
&&-\frac{1}{h^2}\sum\limits_{k=1}^{m-1}\int_0^T\int_{0\leq x^m<h}|\p_k\omega|\.(|\mathcal{T}^m\p_k\p_m\omega-\p_k\theta|x^m+h|\p_k\theta|+\e x^m)\\
&\gtrsim&\int_0^T\int_{\mathbb{S}^+(0,1)}|\tau^m_h\xi|^2|D^m_h\p_x\omega|^2-\int_0^T\int_{\mathbb{S}^+(0,1)}|\p_x\omega|^2-\int_0^T\int_{\mathbb{S}^+(0,1)}|D^m_h\omega|^2\\
&&-\(\frac{1}{h}\int_0^T\int_{0\leq x^m<h}|\p_x\omega|^2\)^{1/2}\sum\limits_{k=1}^{m-1}\(\int_0^T\int|\mathcal{T}^m\p_k\p_m\omega-\p_k\theta|^2+|\p_k\theta|^2+\e^2\)^{1/2}
\end{eqnarray*}
where we have used the fact that $A$ is positive-definite. Furthermore, one are able to get
\begin{eqnarray*}
&&2\int_0^T\int_{\mathbb{S}^+(0,1)}\langle\p_t\omega,D^m_h(\xi^2D^m_{-h}\tilde{\omega})\rangle\\
&=&\int_{\mathbb{S}^+(0,1)}(\tau^m_h\xi)^2|D^m_h\omega|^2(0,\.)-\int_{\mathbb{S}^+(0,1)}(\tau^m_h\xi)^2|D^m_h\omega|^2(T,\.)\\
&&-\frac{2}{h^2}\int_0^T\int_{0\leq x^m<h}\langle\p_t\omega,x^m(\mathcal{T}^m\p_m\omega-\theta)+o(x^m)\rangle-\frac{2}{h}\int_0^T\int_{0\leq x^m<h}\langle\p_t\omega,\theta\rangle,
\end{eqnarray*}
which implies
\begin{eqnarray*}
&&\Big|\int_0^T\int_{\mathbb{S}^+(0,1)}\langle\p_t\omega,D^m_h(\xi^2D^m_{-h}\tilde{\omega})\rangle\Big|\\
&\lesssim&\int_{\mathbb{S}^+(0,1)}(\tau^m_h\xi)^2|D^m_h\omega|^2(0,\.)+\int_{\mathbb{S}^+(0,1)}(\tau^m_h\xi)^2|D^m_h\omega|^2(T,\.)\\
&&+\frac{1}{h^2}\int_0^T\int_{0\leq x^m<h}x^m|\p_t\omega|\(|\mathcal{T}^m\p_m\omega-\theta|+\e\)+\frac{1}{h}\int_0^T\int_{0\leq x^m<h}|\p_t\omega|\.|\theta|\\
&\lesssim&\int_{\mathbb{S}^+(0,1)}|D^m_h\omega|^2(0,\.)+\int_{\mathbb{S}^+(0,1)}|D^m_h\omega|^2(T,\.)\\
&&+\(\frac{1}{h}\int_0^T\int_{0\leq x^m<h}|\p_t\omega|^2\)^{1/2}\(\int_0^T\int|\mathcal{T}^m\p_m\omega-\theta|^2+|\theta|^2+\e^2\)^{1/2}.
\end{eqnarray*}
Classical trick gives
\begin{eqnarray}\label{12}
&&\int_0^T\int_{\mathbb{S}^+(0,1)}|\tau^m_h\xi|^2\.|D^m_h\p_x\omega|^2\\
&\lesssim&\int_{\mathbb{S}^+(0,1)}|D^m_h\omega|^2(0,\.)+\int_{\mathbb{S}^+(0,1)}|D^m_h\omega|^2(T,\.)+\int_{\mathbb{S}^+(0,1)}|D^m_h\omega_0|^2+\int_0^T\int_{\mathbb{S}^+(0,1)}|D^m_h\omega|^2\nonumber\\
&&+\int_0^T\int_{\mathbb{S}^+(0,1)}|\p_x\omega|^2+\int_0^T\int_{\mathbb{S}^+(0,1)}|D^m_hF|^2+\(\mathcal{E}_h\)^{1/2}\mathcal{E}^{1/2},\nonumber
\end{eqnarray}
where
\begin{eqnarray*}
\mathcal{E}_h:=\frac{1}{h}\(\int_{0\leq x^m<h}|\omega_0|^2+\int_{0\leq x^m<h}|\omega(0,\.)|^2+\int_0^T\int_{0\leq x^m<h}|F|^2+\int_0^T\int_{0\leq x^m<h}|\p_x\omega|^2+\int_0^T\int_{0\leq x^m<h}|\p_t\omega|^2\)
\end{eqnarray*}
and
\begin{eqnarray*}
\mathcal{E}:&=&(1+T)\e^2+\int_0^T\int\(|\mathcal{T}^m\p_m\omega-\theta|^2+|\theta|^2\)+\sum\limits_{k=1}^{m-1}\int_0^T\int\(|\mathcal{T}^m\p_k\p_m\omega-\p_k\theta|^2+|\p_k\theta|^2\)\\
&&+\int\(|\mathcal{T}^m\p_m\omega(0,\.)-\theta(0,\.)|^2+|\theta(0,\.)|^2\).
\end{eqnarray*}
Lebesgue differentiation theorem tells us that the right-hand side of \ref{12} is bounded for all $h>0$. Hence, $\p_m\p_x\omega$ exists and we also have
\begin{eqnarray}\label{13'}
\left\{
\begin{array}{rl}
&\p_t\omega-\p_i[A(\p_i\omega)]+B(\p_Y\omega)+C(\omega)=F,\s\mbox{on}\s\mathbb{S}^+(0,1)\\
&\omega(0,\.)=\omega_0,\s\s\mathcal{T}^m\p_m\omega=\theta.
\end{array}
\right.
\end{eqnarray}

In the next, we shall employ induction. Assume that $\p_x^{\alpha}\omega$ exists for all $|\alpha|\leq k$($k\geq 2$). Taking $D^3_h\p_x^{\beta}$ with $|\beta|=k-2$ to act on both sides of $(\ref{13'})$, we obtain
\begin{eqnarray*}
D^3_h\p_t\p_x^{\beta}\omega-\p_iD^3_h\p_x^{\beta}[A(\p_i\omega)]+D^3_h\p_x^{\beta}[B(\p_Y\omega)]+D^3_h\p_x^{\beta}[C(\omega)]=D^3_h\p_x^{\beta}F\s\mbox{on}\s\mathbb{S}^+(0,3/4),
\end{eqnarray*}
where $D_h:=(D^1_h,\cdots,D^m_h)$.

Multipling $\xi^2D^3_h\p_x^{\beta}\omega$ on both sides of the above equation and integrating by parts lead to
\begin{eqnarray*}
&&\frac{1}{2}\frac{d}{dt}\int_{\mathbb{S}^+(0,1)}\xi^2|D^3_h\p_x^{\beta}\omega|^2+\int_{\mathbb{S}^+(0,1)}\langle D^3_h\p_x^{\beta}[A(\p_i\omega)],\p_i\(\xi^2D^3_h\p_x^{\beta}\omega\)\rangle\\
&&+\int_{\mathbb{S}^+(0,1)}\xi^2\langle D^3_h\p_x^{\beta}\omega,D^3_h\p_x^{\beta}[B(\p_Y\omega)]\rangle+\int_{\mathbb{S}^+(0,1)}\xi^2\langle D^3_h\p_x^{\beta}\omega,D^3_h\p_x^{\beta}[C(\omega)]\rangle\\
&=&\int_{\mathbb{S}^+(0,1)}\xi^2\langle D^3_h\p_x^{\beta}\omega,D^3_h\p_x^{\beta}F\rangle,
\end{eqnarray*}
where $\xi\in C_c^{\infty}(\mathbb{S}^+(0,3/4))$. Since $A$ is positive-definite, we obtain Gronwall's Inequality
\begin{eqnarray*}
\frac{d}{dt}\int_{\mathbb{S}^+(0,1)}\xi^2|D^3_h\p_x^{\beta}\omega|^2\lesssim\int_{\mathbb{S}^+(0,1)}\xi^2|D^3_h\p_x^{\beta}\omega|^2+\sum\limits_{\substack{|\gamma|\leq k-2,\,p\leq 3\\|\gamma|+p\leq k}}\int_{\mathbb{S}^+(0,1)}|D^p_h\p_x^{\gamma}\omega|^2+\int_{\mathbb{S}^+(0,1)}|D^3_h\p_x^{\beta}F|^2,
\end{eqnarray*}
which implies that $\p_x^{\alpha}\omega$ exists for $|\alpha|=k+1$.

The uniqueness is easy and we leave it to the readers.
\end{proof}
\section{Proof of the quasi-linear system}\label{section3}
\begin{proof}
In this subsection, we denote the smallest integer not smaller than $m/a-m/p$ by $\lceil p\rceil_a$. Note $\lceil p\rceil_a\leq\lceil m/a\rceil$, where $\lceil x\rceil$ means the smallest integer not smaller than $x$.

Choose $\Omega_0(t,x):=\omega_0(x)$. For $n\geq 1$, let $\Omega_n$ be the unqiue smooth solution of the following linear system
\begin{eqnarray}\label{eq:18'}
\left\{
\begin{array}{rl}
&\D_t\Omega_n+\L_A\Omega_n=F(t,x,\Omega_{n-1},\D_x\Omega_{n-1})\\
&\Omega_n(0,\.)=\omega_0,\s\D_{\vec{n}}\Omega_n=\theta.
\end{array}
\right.
\end{eqnarray}
Defining $\alpha_n:=\Omega_{n+1}-\Omega_{n}$,  we have
\begin{eqnarray}\label{eq:5}
\left\{
\begin{array}{rl}
&\D_t\alpha_n+\L_A\alpha_n=\int_0^1ds(D_{\eta}F)(t,x,s\alpha_{n-1}+\Omega_{n-1},\D_x\Omega_{n-1})(\alpha_{n-1})\\
&+\int_0^1ds(D_{\rho}F)(t,x,\Omega_{n-1},s\D_x\alpha_{n-1}+\D_x\Omega_{n-1})(\D_x\alpha_{n-1})\\
&:=G(t,x,\Omega_{n-1},\D_x\Omega_{n-1},\alpha_{n-1},\D_x\alpha_{n-1})\\
&\alpha_n(0,\.)=0,\s\D_{\vec{n}}\alpha_n=0.
\end{array}
\right.
\end{eqnarray}

\textbf{Step 1: Estimating $\Omega_n$}

Taking $k$-times derivatives with respect to $t$ on both sides of $(\ref{eq:18'})$ yields
\begin{eqnarray}\label{eq:3}
\left\{
\begin{array}{rl}
&\D^{k+1}_t\Omega_n+\L_A(\D^k_t\Omega_n)=\hat{F}^{n-1}_k\\
&\Omega_n(0,\.)=\omega_0,\s\D^j_t\Omega_n(0,\.)=-\L_A[\D_t^{j-1}\Omega_n(0,\.)]+\hat{F}^{n-1}_{j-1}|_{t=0},\s\s\mbox{for}\s 1\leq j\leq k\\
&\D_{\vec{n}}\D^k_t\Omega_n=\D^k_t\theta
\end{array}
\right.
\end{eqnarray}
where $\hat{F}^{n-1}_k:=\hat{F}^{n-1}_k(t,x,\Omega_{n-1},\D_t\Omega_{n-1},\cdots,\D_t^k\Omega_{n-1},\D_x\Omega_{n-1},\D_x\D_t\Omega_{n-1},\cdots,\D_x\D_t^k\Omega_{n-1})$. Tedious computation gives
\begin{eqnarray*}
|\hat{F}^{n-1}_k|\leq\sum C_{p0qr}(|\Omega_{n-1}|,|\D_x\Omega_{n-1}|)\prod\limits_{i=1}^q|\D_t^{q_i}\Omega_{n-1}|^{\tau_i}\prod\limits_{j=1}^r|\D_t^{r_j}\D_x\Omega_{n-1}|^{\mu_j}
\end{eqnarray*}
with
$$
p+\sum\limits_{i=1}^qq_i\tau_i+\sum\limits_{j=1}^rr_j\mu_j=k.
$$
Defining
\begin{eqnarray*}
\mathcal{E}_n^k(t):=\frac{1}{J}\sum\limits_{i=0}^{k}\int_M\Big|\D_t^i\Omega_n\Big|^J,
\end{eqnarray*}
we have
\begin{eqnarray*}
\frac{d}{dt}\mathcal{E}_n^k&=&\sum\limits_{i=0}^{k-1}\int\Big|\D_t^i\Omega_n\Big|^{J-2}\langle\D_t^{i+1}\Omega_n,\D_t^{i}\Omega_n\rangle+\underbrace{\int\Big|\D_t^{k}\Omega_n\Big|^{J-2}\langle\D_t^{k+1}\Omega_n,\D_t^{k}\Omega_n\rangle}_{I}.
\end{eqnarray*}
Now we compute $I$
\begin{eqnarray*}
I&=&\int\Big|\D_t^{k}\Omega_n\Big|^{J-2}\langle\hat{F}^{n-1}_k-\L_A(\D_t^k\Omega_n),\D_t^{k}\Omega_n\rangle\\
&=&\int\Big|\D_t^{k}\Omega_n\Big|^{J-2}\langle\hat{F}^{n-1}_k,\D_t^{k}\Omega_n\rangle-\int\Big|\D_t^{k}\Omega_n\Big|^{J-2}\langle A(\D_x\D_t^k\Omega_n),\D_x\D_t^{k}\Omega_n\rangle\sharp g^{-1}\\
&&-(J-2)\int\Big|\D_t^{k}\Omega_n\Big|^{J-4}\langle A(\D_x\D_t^k\Omega_n),\D_t^{k}\Omega_n\rangle\otimes\langle \D_x\D_t^k\Omega_n,\D_t^{k}\Omega_n\rangle\sharp g^{-1}\\
&&-\int_{\p M}\Big|\D_t^{k}\Omega_n\Big|^{J-2}\langle A(\D_t^k\theta),\D_t^{k}\Omega_n\rangle.
\end{eqnarray*}
\textbf{Condition 2} implies
\begin{eqnarray*}
I\lesssim\int\Big|\D_t^{k}\Omega_n\Big|^{J-1}|\hat{F}^{n-1}_k|-\int\Big|\D_t^{k}\Omega_n\Big|^{J-2}|\D_x\D_t^k\Omega_n|^2+||\D_t^k\theta||^J_{L^J(\p M)}+||\D_t^k\Omega_n||^J_{L^J(\p M)}.
\end{eqnarray*}
Boundary Trace Theorem tells us
\begin{eqnarray*}
||\D_t^k\Omega_n||^J_{L^J(\p M)}&=&||\,|\D_t^k\Omega_n|^J||_{L^1(\p M)}\lesssim||\,|\D_t^k\Omega_n|^J||_{W_1^{1,0}}\lesssim||\D_t^k\Omega_n||^J_{L^J}+||\,|\D_t^k\Omega_n|^{J-1}|\D_x\D_t^k\Omega_n|\,||_{L^1}\\
&\lesssim&C(\epsilon)||\D_t^k\Omega_n||^J_{L^J}+\epsilon||\,|\D_t^k\Omega_n|^{J-2}|\D_x\D_t^k\Omega_n|^2\,||_{L^1}
\end{eqnarray*}
where we have used Kato's Inequality and H\"{o}lder's Inequality.

Combining the above inequalities leads to
\begin{eqnarray*}
\frac{d}{dt}\mathcal{E}_n^k\leq C\mathcal{E}_n^k+C||\hat{F}^{n-1}_k||^J_{L^J}+C||\D_t^k\theta||^J_{L^J(\p M)}.
\end{eqnarray*}

Assume $C_{psqr}(|\eta|,|\rho|)=\sum\limits_{w+u=1}^{\infty}C_{wu}^{psqr}|\eta|^w|\rho|^u$. H\"{o}lder's Inequality implies
\begin{eqnarray}\label{jzl}
||\hat{F}^{n-1}_k||_{L^J}\leq\sum||C_{p0qr}(|\Omega_{n-1}|,|\D_x\Omega_{n-1}|)||_{L^{\chi}}\prod\limits_{i=1}^q||\D_t^{q_i}\Omega_{n-1}||^{\tau_i}_{L^{a_i\tau_i}}\prod\limits_{j=1}^r||\D_t^{r_j}\D_x\Omega_{n-1}||^{\mu_j}_{L^{b_j\mu_j}}
\end{eqnarray}
with
$$
p+\sum\limits_{i=1}^qq_i\tau_i+\sum\limits_{j=1}^rr_j\mu_j=k\s\s\mbox{and}\s\s\frac{1}{\chi}+\sum\limits_{i=1}^q\frac{1}{a_i}+\sum\limits_{j=1}^r\frac{1}{b_j}=\frac{1}{J},
$$
where $\chi$, $a_i$ and $b_j$ are to be determined. General Minkowski's Inequality tells us that
$$
||C_{psqr}(|\Omega_{n-1}|,|\D_x\Omega_{n-1}|)||_{L^{\chi}}\leq\sum\limits_{w+u=1}^{\infty}|C_{wu}^{psqr}|\.||\,\,|\Omega_{n-1}|^w\.|\D_x\Omega_{n-1}|^u||_{L^{\chi}}.
$$
Now we choose $c_w$ and $d_u$ which are to be determined such that $1/c_w+1/d_u=1/\chi$
$$
||\,\,|\Omega_{n-1}|^w\.|\D_x\Omega_{n-1}|^u||_{L^{\chi}}\leq||\Omega_{n-1}||^w_{L^{w\.c_w}}||\D_x\Omega_{n-1}||^u_{L^{u\.d_u}}.
$$

In the next, we shall estimate $||\D_t^p\Omega_{n-1}||_{W^{k,0}_J}$. From Lemma \ref{lemma2.1} and $(\ref{eq:3})$ it follows that
\begin{eqnarray}\label{11}
||\D_t^p\Omega_{n-1}||_{W^{k,0}_J}&\leq& C(k-2,J)\{||\theta||_{W^{k-1,0}_J}+||\hat{F}_p^{n-1}||_{W^{k-2,0}_J}\\
&&+||\D_t^{p+1}\Omega_{n-1}||_{W^{k-1,0}_J}+(1+||\D_t^p\Omega_{n-1}||_{W^{k-1,0}_J})^{2k-3}\}.\nonumber
\end{eqnarray}
It is easy to see that the key is to control $||\hat{F}_p^{n-1}||_{W^{k-2,0}_J}$. Tedious computation gives
\begin{eqnarray*}
|\D_x^j\hat{F}^{n-1}_p|&\leq&\sum C_{abqr}(|\Omega_{n-1}|,|\D_x\Omega_{n-1}|)\prod\limits_{i=1}^q|\D_x^{c_i}\D_t^{d_i}\Omega_{n-1}|^{\tau_i}\prod\limits_{l=1}^r|\D_x^{1+e_l}\D_t^{h_l}\Omega_{n-1}|^{\mu_l}\\
&\leq&\sum\sum\limits_{w+u=1}^{\infty}|C^{abqr}_{wu}|\.|\Omega_{n-1}|^w|\D_x\Omega_{n-1}|^u\prod\limits_{i=1}^q|\D_x^{c_i}\D_t^{d_i}\Omega_{n-1}|^{\tau_i}\prod\limits_{l=1}^r|\D_x^{1+e_l}\D_t^{h_l}\Omega_{n-1}|^{\mu_l}
\end{eqnarray*}
with
$$
a+\sum\limits_{i=1}^qd_i\tau_i+\sum\limits_{l=1}^rh_l\mu_l=p\s\mbox{and}\s b+\sum\limits_{i=1}^qc_i\tau_i+\sum\limits_{l=1}^r(1+e_l)\mu_l=j.
$$
H\"{o}lder's Inequality and general Minkowski's Inequality lead to
$$
||\D_x^j\hat{F}^{n-1}_p||_{L^J}\leq\sum\sum\limits_{w+u=1}^{\infty}|C^{abqr}_{wu}|\.||\Omega_{n-1}||_{L^{w\nu_w}}^w||\D_x\Omega_{n-1}||_{L^{u\eta_u}}^u\prod\limits_{i=1}^q||\D_x^{c_i}\D_t^{d_i}\Omega_{n-1}||_{L^{\delta_i\tau_i}}^{\tau_i}\prod\limits_{l=1}^r||\D_x^{1+e_l}\D_t^{h_l}\Omega_{n-1}||_{L^{\alpha_l\mu_l}}^{\mu_l}
$$
with $1/\nu_w+1/\eta_u+\sum\limits_{i=1}^q1/\delta_i+\sum\limits_{l=1}^r1/\alpha_l=1/J$ to be determined. Sobolev's Embedding tells us that there exists a constant $c(p)$ such that
$$
||\omega||_{L^p}\leq c(p)||\omega||_{W^{\lceil p\rceil_J,0}_J}
$$
where $c(p)$ is non-decreasing with respect to $p$ with $c(J)=1$. Hence it follows that
\begin{eqnarray}\label{10}
||\D_x^j\hat{F}^{n-1}_p||_{L^J}&\leq&\sum\sum\limits_{w+u=1}^{\infty}c(w\nu_w)^wc(u\eta_u)^u\prod\limits_{i=1}^qc(\delta_i\tau_i)^{\tau_i}\prod\limits_{l=1}^rc(\alpha_l\mu_l)^{\mu_l}|C^{abqr}_{wu}|\.||\Omega_{n-1}||^w_{W^{\lceil w\nu_w\rceil_J,0}_J}||\Omega_{n-1}||_{W^{1+\lceil u\eta_u\rceil_J,0}_J}^u\nonumber\\
&&\prod\limits_{i=1}^q||\D_t^{d_i}\Omega_{n-1}||_{W^{c_i+\lceil\delta_i\tau_i\rceil_J,0}_J}^{\tau_i}\prod\limits_{l=1}^r||\D_t^{h_l}\Omega_{n-1}||_{W^{1+e_l+\lceil\alpha_l\mu_l\rceil_J,0}_J}^{\mu_l}
\end{eqnarray}
Substituting $(\ref{10})$ into $(\ref{11})$ yields
\begin{eqnarray}\label{1.14}
&&||\D_t^p\Omega_{n-1}||_{W^{k,0}_J}\leq C(k,J)\left\{||\D_t^{p+1}\Omega_{n-1}||_{W^{k-2,0}_J}+(1+||\D_t^p\Omega_{n-1}||_{W^{k-1,0}_J})^{2k-3}+||\theta||_{W^{k-1,0}_J}\right\}\\
&&+C(k,J)\sum\limits_{j=0}^{k-2}\sum\limits_{\substack{a+\sum\limits_{i=1}^qd_i\tau_i+\sum\limits_{l=1}^rh_l\mu_l=p\\b+\sum\limits_{i=1}^qc_i\tau_i+\sum\limits_{l=1}^r(1+e_l)\mu_l=j}}\sum\limits_{w+u=1}^{\infty}c(w\nu_w)^wc(u\eta_u)^u\prod\limits_{i=1}^qc(\delta_i\tau_i)^{\tau_i}\prod\limits_{l=1}^rc(\alpha_l\mu_l)^{\mu_l}|C^{abqr}_{wu}|\nonumber\\
&&||\Omega_{n-1}||^w_{W^{\lceil w\nu_w\rceil_J,0}_J}||\Omega_{n-1}||_{W^{1+\lceil u\eta_u\rceil_J,0}_J}^u\prod\limits_{i=1}^q||\D_t^{d_i}\Omega_{n-1}||_{W^{c_i+\lceil\delta_i\tau_i\rceil_J,0}_J}^{\tau_i}\prod\limits_{l=1}^r||\D_t^{h_l}\Omega_{n-1}||_{W^{1+e_l+\lceil\alpha_l\mu_l\rceil_J,0}_J}^{\mu_l}.\nonumber
\end{eqnarray}
Define $\N_k^p(n):=||\Omega_{n}||_{W^{k,p}_J}$, which is increasing with respect to $k$ and $p$, and choose
$$
\nu_w=\eta_u=\delta_i=\alpha_l:=J(2+q+r)\leq J(k+p).
$$
Moreover, $(\ref{1.14})$ implies
\begin{eqnarray*}
&&\N^p_k(n-1)\leq C(k,J)\left\{\N^{p+1}_{k-2}(n-1)+[1+\N^p_{k-1}(n-1)]^{2k-3}+||\theta||_{W^{k-1,0}_J}\right\}\\
&&+2^{k+p-2}C(k,J)c(J(k+p)^2)^{k+p}\N^{p}_{k-1}\left(n-1\right)^{\min\{k-2,p\}}\\
&&\sum\limits_{w+u=1}^{\infty}c(J(w+u)(k+p))^{w+u}\left(\sum\limits_{a+b+q+r=0}^{k-2+p}|C^{abqr}_{wu}|\right)\N^0_2\left(n-1\right)^{w+u}.
\end{eqnarray*}
Under \textbf{Condition 1} we obtain
\begin{eqnarray}\label{jxw}
\N^p_k(n-1)&\leq& C(k,J)2^{4k-7+2p}\left[1+\N^{p+1}_{k-1}(n-1)+||\theta||_{W^{k-1,0}_J}\right]^{3k-5+p}\\
&&c(J(k+p)^2)^{k+p}\ex\left\{(k+p)\N^0_{2}(n-1)\right\},\nonumber
\end{eqnarray}
where $\ex(x):=[\exp(x)-1]/L$.

Now our goal is to estimate $||\hat{F}^{n-1}_k||_{L^J}$ in $(\ref{jzl})$
\begin{eqnarray*}
&&||\hat{F}^{n-1}_k||_{L^J}\leq\sum\limits_{p+\sum\limits_{i=1}^qq_i\tau_i+\sum\limits_{j=1}^rr_j\mu_j=k}\sum\limits_{w+u=1}^{\infty}|C^{p0qr}_{wu}|c(w\.c_w)^wc(u\.d_u)^u\prod\limits_{i=1}^qc(a_i\tau_i)^{\tau_i}\prod\limits_{j=1}^rc(b_j\mu_j)^{\mu_j}\\
&&[\N^0_{\lceil wc_w\rceil_J}(n-1)]^w[\N^0_{1+\lceil ud_u\rceil_J}(n-1)]^u\prod\limits_{i=1}^q[\N^{q_i}_{\lceil a_i\tau_i\rceil_J}(n-1)]^{\tau_i}\prod\limits_{j=1}^r[\N^{r_j}_{1+\lceil b_j\mu_j\rceil_J}(n-1)]^{\mu_j}\\
&&\leq\sum\limits_{p+\sum\limits_{i=1}^qq_i\tau_i+\sum\limits_{j=1}^rr_j\mu_j=k}\sum\limits_{w+u=1}^{\infty}|C^{p0qr}_{wu}|c(w\.c_w)^wc(u\.d_u)^u\prod\limits_{i=1}^qc(a_i\tau_i)^{\tau_i}\prod\limits_{j=1}^rc(b_j\mu_j)^{\mu_j}\\
&&[\N^0_{1}(n-1)]^w[\N^0_{2}(n-1)]^u\prod\limits_{i=1}^q[\N^{q_i}_{1}(n-1)]^{\tau_i}\prod\limits_{j=1}^r[\N^{r_j}_{2}(n-1)]^{\mu_j}\\
&&\leq\sum\limits_{w+u=1}^{\infty}\left(\sum\limits_{a+b+c=0}^k|C^{a0bc}_{wu}|\right)\sum\limits_{p+\sum\limits_{i=1}^qq_i\tau_i+\sum\limits_{j=1}^rr_j\mu_j=k}c(w\.c_w)^wc(u\.d_u)^u\\
&&\prod\limits_{i=1}^qc(a_i\tau_i)^{\tau_i}\prod\limits_{j=1}^rc(b_j\mu_j)^{\mu_j}[\N^0_{2}(n-1)]^{u+w}[\N^{k}_{2}(n-1)]^{k}.
\end{eqnarray*}
Choosing $c_w=d_u:=2J(1+q+r)\leq 2J(1+k)$ and $a_i=b_j:=J(1+q+r)\leq J(1+k)$ implies
\begin{eqnarray*}
&&||\hat{F}^{n-1}_k||_{L^J}\leq2^k\sum\limits_{w+u=1}^{\infty}\left(\sum\limits_{f+h+l=0}^k|C^{f0hl}_{wu}|\right)c(2J(1+k)(w+u))^{w+u}\\
&&c(J(1+k)k)^{k}[\N^0_{2}(n-1)]^{u+w}[\N^{k}_{2}(n-1)]^{k}.
\end{eqnarray*}
Under \textbf{Condition 1} we see
\begin{eqnarray}\label{jxw2}
||\hat{F}^{n-1}_k||_{L^J}\leq2^k\ex\{2(1+k)\N_2^0(n-1)\}c(J(1+k)k)^{k}[\N^{k}_{2}(n-1)]^{k}.
\end{eqnarray}
$(\ref{jxw})$ tells us
\begin{eqnarray*}
&&\N^{k-1}_{2}(n-1)\leq c(J(1+k)^2)^{1+k}C(2,J)2^{2k-1}[1+\N^{k}_1(n-1)+||\theta||_{W^{1,0}_J}]^{k}\ex\{(1+k)\N_2^0(n-1)\}\\
&&\leq c(J(1+k)^2)^{2k^2}C(2,J)^{k+1}2^{4k^2}[1+J^{1/J}(\mathcal{E}^{k+1}_{n-1})^{1/J}+||\theta||_{W^{1,0}_J}]^{k^2}\ex\{3k^2\N_2^0(n-1)\}\\
&&:=V(k,\mathcal{E}^{k+1}_{n-1},\N_2^0(n-1)).
\end{eqnarray*}
Substituting the above inequalities into $(\ref{jxw2})$ gives
\begin{eqnarray}\label{ldy}
&&||\hat{F}^{n-1}_k||_{L^J}\leq 2^k\ex\left\{2(1+k)\N_2^0(n-1)\right\}c(J(1+k)k)^{k}V(k,\mathcal{E}^{k+1}_{n-1},\N_2^0(n-1))^{k},
\end{eqnarray}
which implies
\begin{eqnarray*}
\frac{d}{dt}\mathcal{E}_n^k&\leq&C\mathcal{E}_n^{k}+C\.H(k,\mathcal{E}_{n-1}^{k+1},\N_2^0(n-1)),
\end{eqnarray*}
where 
\begin{eqnarray*}
&&H(k,\mathcal{E}_{n-1}^{k+1},\N_2^0(n-1)):=||\D_t^k\theta||^J_{L^J(\p M)}\\
&&+2^{Jk}\ex\left\{2J(1+k)\N_2^0(n-1)\right\}c(J(1+k)k)^{Jk}V(k,\mathcal{E}^{k+1}_{n-1},\N_2^0(n-1))^{Jk}.
\end{eqnarray*}
Gronwall's Inequality tells us
\begin{eqnarray}\label{J8}
\mathcal{E}_n^k(t)&\leq&\exp(Ct)\mathcal{E}_n^k(0)+\int_0^t\exp\{C(t-s)\}H(k,\mathcal{E}_{n-1}^{k+1}(s),\N_2^0(n-1)(s))\,ds.
\end{eqnarray}

In the next, we shall estimate $\mathcal{E}_n^k(0)$. Induction argument leads to
$$
(\D_t^k\Omega_n)(0,\.):=F^n_k:=\sum\limits_{p+q=k-1}(-\L_A)^p[\hat{F}^{n-1}_q(0,\.,F^{n-1}_0,\cdots,F_q^{n-1},\D_x F^{n-1}_0,\cdots,\D_x F_q^{n-1})]\s\mbox{with}\s F_0^n=\omega_0.
$$
Hence,
\begin{eqnarray*}
\mathcal{E}_n^k(0)=\frac{1}{J}\sum\limits_{r=0}^k||F^n_r||_{L^J}^J\leq\frac{1}{J}||\omega_0||^J_{L^J}+\frac{1}{J}\sum\limits_{r=1}^kr^{J-1}\sum\limits_{p+q=r-1}||\L_A^p[\hat{F}^{n-1}_q(0,\.,F^{n-1}_0,\cdots,F_q^{n-1},\D_x F^{n-1}_0,\cdots,\D_x F_q^{n-1})]||^J_{L^J}.
\end{eqnarray*}
Since $A$ is smooth on $[0,T]\times M$, we have
\begin{eqnarray}\label{sl}
||\D_x^{\kappa}\L_A^p\omega||_{L^J}\leq C(||A||_{W^{\kappa+2p,0}_{\infty}})||\omega||_{W^{\kappa+2p,0}_J}\s\s\mbox{for any}\s\omega\in E.
\end{eqnarray}
Hence,
\begin{eqnarray}\label{J}
\s\s\s||F^n_r||^J_{W^{\kappa,0}_J}\leq r^{J-1}C(||A||_{W^{\kappa+2r,0}_{\infty}})^J\sum\limits_{q=0}^{r-1}||\hat{F}^{n-1}_q(0,\.,F^{n-1}_0,\cdots,F_q^{n-1},\D_x F^{n-1}_0,\cdots,\D_x F_q^{n-1})||^J_{W^{\kappa+2r-2,0}_J}.
\end{eqnarray}
However,
\begin{eqnarray*}
&&|\D_x^j[\hat{F}^{n-1}_p(0,\.,F^{n-1}_0,\cdots,F_p^{n-1},\D_x F^{n-1}_0,\cdots,\D_x F_p^{n-1})]|\\
&\leq&\sum C_{abqr}(|\omega_0|,|\D_x\omega_0|)\prod\limits_{i=1}^q|\D_x^{c_i}F_{d_i}^{n-1}|^{\tau_i}\prod\limits_{l=1}^r|\D_x^{1+e_l}F_{h_l}^{n-1}|^{\mu_l}
\end{eqnarray*}
with
$$
a+\sum\limits_{i=1}^qd_i\tau_i+\sum\limits_{l=1}^rh_l\mu_l=p\s\mbox{and}\s b+\sum\limits_{i=1}^qc_i\tau_i+\sum\limits_{l=1}^r(1+e_l)\mu_l=j.
$$
H\"{o}lder's Inequality and Sobolev's Embedding give
\begin{eqnarray}\label{J2}
&&\s\s\s\s||\D_x^j[\hat{F}^{n-1}_p(0,\.,F^{n-1}_0,\cdots,F_p^{n-1},\D_x F^{n-1}_0,\cdots,\D_x F_p^{n-1})]||_{L^J}\\
&\leq&\sum ||C_{abqr}(|\omega_0|,|\D_x\omega_0|)||_{L^{\infty}}\prod\limits_{i=1}^qc(J(q+r)\tau_i)^{\tau_i}||F_{d_i}^{n-1}||^{\tau_i}_{W_J^{c_i+\lceil \delta_i\tau_i\rceil_J,0}}\prod\limits_{l=1}^rc(J(q+r)\mu_l)^{\mu_l}||F_{h_l}^{n-1}||^{\mu_l}_{W_J^{1+e_l+\lceil \alpha_l\mu_l\rceil_J,0}}\nonumber\\
&\leq&c(J(p+j)^2)^{p+j}\sum ||C_{abqr}(|\omega_0|,|\D_x\omega_0|)||_{L^{\infty}}\prod\limits_{i=1}^q||F_{d_i}^{n-1}||^{\tau_i}_{W_J^{c_i+1,0}}\prod\limits_{l=1}^r||F_{h_l}^{n-1}||^{\mu_l}_{W_J^{2+e_l,0}}\nonumber
\end{eqnarray}
where
$$
\delta_i=\alpha_l:=J(q+r).
$$
Substituting $(\ref{J2})$ into $(\ref{J})$ yields
\begin{eqnarray}\label{J3}
&&||F^n_r||^J_{W^{\kappa,0}_J}\leq r^{J-1}C(||A||_{W^{\kappa+2r,0}_{\infty}})^J\sum\limits_{q=0}^{r-1}\sum\limits_{j=0}^{\kappa+2r-2}2^{(q+j)J}c(J(q+j)^2)^{J(q+j)}\\
&&\sum\limits_{\substack{a+\sum\limits_{i=1}^fd_i\tau_i+\sum\limits_{l=1}^uh_l\mu_l=q\\b+\sum\limits_{i=1}^fc_i\tau_i+\sum\limits_{l=1}^u(1+e_l)\mu_l=j}}||C_{abfu}(|\omega_0|,|\D_x\omega_0|)||_{L^{\infty}}^J\prod\limits_{i=1}^f||F_{d_i}^{n-1}||^{J\tau_i}_{W_J^{j+1,0}}\prod\limits_{l=1}^u||F_{h_l}^{n-1}||^{J\mu_l}_{W_J^{1+j,0}}\nonumber.
\end{eqnarray}
Define $\mathcal{E}_n^{\kappa,k}:=\sum\limits_{r=0}^k||F^n_r||^J_{W^{\kappa,0}_J}$. $(\ref{J3})$ tells us
\begin{eqnarray}\label{J5}
1+\mathcal{E}_n^{\kappa,k}&\leq& k^{J+1}C(||A||_{W^{\kappa+2k,0}_{\infty}})^J\sum\limits_{q=0}^{k-1}\sum\limits_{j=0}^{\kappa+2k-2}2^{(q+j)J}c(J(q+j)^2)^{J(q+j)}\nonumber\\
&&\left(\sum\limits_{a+b+f+u=0}^{q+j}||C_{abfu}(|\omega_0|,|\D_x\omega_0|)||_{L^{\infty}}^J\right)\left(1+\mathcal{E}_{n-1}^{j+1,q}\right)^{k-1}\nonumber\\
&\leq&(\kappa+2k)k^{J+2}C(||A||_{W^{\kappa+2k,0}_{\infty}})^J2^{(\kappa+3k-3)J}c(J(\kappa+3k-3)^2)^{J(\kappa+3k-3)}\nonumber\\
&&\left(\sum\limits_{a+b+f+u=0}^{\kappa+3k-3}||C_{abfu}(|\omega_0|,|\D_x\omega_0|)||_{L^{\infty}}^J\right)\left(1+\mathcal{E}_{n-1}^{\kappa+2k-1,k-1}\right)^{k-1}.
\end{eqnarray}
From $\mathcal{E}^k_n(0)=\mathcal{E}_n^{0,k}/J$ it follows that
\begin{eqnarray*}
\mathcal{E}^k_n(0)&\leq& k^{J+3}C(||A||_{W^{2k,0}_{\infty}})^J2^{(3k-3)J}c(J(3k-3)^2)^{J(3k-3)}\\
&&\left(\sum\limits_{a+b+f+u=0}^{3k-3}||C_{abfu}(|\omega_0|,|\D_x\omega_0|)||_{L^{\infty}}^J\right)\left(1+\mathcal{E}_{n-1}^{2k-1,k-1}\right)^{k-1}.
\end{eqnarray*}
Defining $S(k):=\sum\limits_{a+b+f+u=0}^{k}||C_{abfu}(|\omega_0|,|\D_x\omega_0|)||_{L^{\infty}}^J$ leads to
\begin{eqnarray*}
\mathcal{E}^k_n(0)\leq k^{J+3}C(||A||_{W^{2k,0}_{\infty}})^J2^{(3k-3)J}c(J(3k-3)^2)^{J(3k-3)}S(3k-3)\left(1+\mathcal{E}_{n-1}^{2k-1,k-1}\right)^{k-1}.
\end{eqnarray*}
However, $(\ref{J5})$ gives
\begin{eqnarray*}
&&\frac{\log\left(1+\mathcal{E}_{n}^{\kappa,k}\right)}{k!}-(J+2)\frac{\log k}{k!}-\frac{\log(\kappa+2k)}{k!}-J\frac{\log C(||A||_{W^{\kappa+2k,0}_{\infty}})}{k!}-\frac{\kappa+3k-3}{k!}J\log2\\
&&-\frac{\log\left(S(\kappa+3k-3)\right)}{k!}-J(\kappa+3k-3)\frac{\log\left[c(J(\kappa+3k-3)^2)\right]}{k!}\leq\frac{\log\left(1+\mathcal{E}_{n-1}^{\kappa+2k-1,k-1}\right)}{(k-1)!}.
\end{eqnarray*}
From the definition it follows that $\mathcal{E}_0^{\kappa,k}=0$. Hence, if $n\leq k-1$, we have
\begin{eqnarray*}
&&\frac{\log\left(1+\mathcal{E}_{n}^{\kappa,k}\right)}{k!}\leq(J+2)\sum\limits_{j=0}^n\frac{\log(k-j)}{(k-j)!}+J\sum\limits_{j=0}^n\frac{\log C(||A||_{W^{\kappa+(2k-3)j+2k,0}_{\infty}})}{(k-j)!}\\
&&+J\log2\sum\limits_{j=0}^n\frac{\kappa+(2k-4)j+3k-3}{(k-j)!}+\sum\limits_{j=0}^n\frac{\log\left(S(\kappa+(2k-4)j+3k-3)\right)}{(k-j)!}\\
&&+\sum\limits_{j=0}^n\frac{\log[\kappa+(2k-3)j+2k]}{(k-j)!}+J\log2\sum\limits_{j=0}^n\frac{\kappa+(2k-4)j+3k-3}{(k-j)!}\\
&&+\sum\limits_{j=0}^nJ[\kappa+(2k-4)j+3k-3]\frac{\log\{c(J[\kappa+(2k-4)j+3k-3]^2)\}}{(k-j)!};
\end{eqnarray*}
otherwise, one are able to obtain
\begin{eqnarray}\label{wcx}
&&\frac{\log\left(1+\mathcal{E}_{n}^{\kappa,k}\right)}{k!}\leq(J+2)\sum\limits_{j=0}^{k-1}\frac{\log(k-j)}{(k-j)!}+J\sum\limits_{j=0}^{k-1}\frac{\log C(||A||_{W^{\kappa+(2k-3)j+2k,0}_{\infty}})}{(k-j)!}\\
&&+J\log2\sum\limits_{j=0}^{k-1}\frac{\kappa+(2k-4)j+3k-3}{(k-j)!}+\sum\limits_{j=0}^{k-1}\frac{\log\left(S(\kappa+(2k-4)j+3k-3)\right)}{(k-j)!}\nonumber\\
&&+\sum\limits_{j=0}^{k-1}\frac{\log[\kappa+(2k-3)j+2k]}{(k-j)!}+J\log2\sum\limits_{j=0}^{k-1}\frac{\kappa+(2k-4)j+3k-3}{(k-j)!}\nonumber\\
&&+\sum\limits_{j=0}^{k-1}J[\kappa+(2k-4)j+3k-3]\frac{\log\{c(J[\kappa+(2k-4)j+3k-3]^2)\}}{(k-j)!}\nonumber\\
&&+\log\left[1+\mathcal{E}_{n-k+1}^{\kappa+(2k-1)(k-1),1}\right].\nonumber
\end{eqnarray}
Thanks to $(\ref{J})$ and the definition of $\mathcal{E}_{n}^{\kappa,k}$, we obtain
$$
\mathcal{E}_{n}^{\kappa,1}\leq||\omega_0||^J_{W^{\kappa,0}_J}+C(||A||_{W_{\infty}^{\kappa+2,0}})^J||F(0,\.,\omega_0,\D_x\omega_0)||^J_{W^{\kappa,0}_J}:=\lambda(\kappa).
$$
Substituting the above upper bound into $(\ref{wcx})$ yields
\begin{eqnarray}\label{wcx2}
&&\frac{\log\left(1+\mathcal{E}_{n}^{\kappa,k}\right)}{k!}\leq(J+2)\sum\limits_{j=0}^{k-1}\frac{\log(k-j)}{(k-j)!}+J\sum\limits_{j=0}^{k-1}\frac{\log C(||A||_{W^{\kappa+(2k-3)j+2k,0}_{\infty}})}{(k-j)!}\\
&&+J\log2\sum\limits_{j=0}^{k-1}\frac{\kappa+(2k-4)j+3k-3}{(k-j)!}+\sum\limits_{j=0}^{k-1}\frac{\log\left(S(\kappa+(2k-4)j+3k-3)\right)}{(k-j)!}\nonumber\\
&&+\sum\limits_{j=0}^{k-1}\frac{\log[\kappa+(2k-3)j+2k]}{(k-j)!}+J\log2\sum\limits_{j=0}^{k-1}\frac{\kappa+(2k-4)j+3k-3}{(k-j)!}\nonumber\\
&&+\sum\limits_{j=0}^{k-1}J[\kappa+(2k-4)j+3k-3]\frac{\log\{c(J[\kappa+(2k-4)j+3k-3]^2)\}}{(k-j)!}\nonumber\\
&&+\log\left[1+\lambda(\kappa+(2k-1)(k-1))\right]:=\mu(\kappa,k).\nonumber
\end{eqnarray}
In conclusion, $(\ref{wcx2})$ always holds true whether $n\leq k-1$ or not, which implies
$$
\mathcal{E}^k_n(0)\leq\exp\{\mu(0,k)k!\}/J.
$$
\begin{remark}
Recalling the definition of $S(k)$ we know that if $S((2k-1)(k-1))<<1$, $\mu(0,k)$ may be negative.
\end{remark}

\textbf{Step 2: Estimating $\alpha_n$}

Taking $k$-times derivatives with respect to $t$ on both sides of $(\ref{eq:5})$ yields
\begin{eqnarray}\label{eq:3'}
\left\{
\begin{array}{rl}
&\D^{k+1}_t\alpha_n+\L_A(\D^k_t\alpha_n)=\hat{G}^{n-1}_k\\
&\alpha_n(0,\.)=0,\s\D^j_t\alpha_n(0,\.)=-\L_A[\D_t^{j-1}\alpha_n(0,\.)]+\hat{G}^{n-1}_{j-1}|_{t=0},\s\s\mbox{for}\s 1\leq j\leq k\\
&\D_{\vec{n}}\D^k_t\alpha_n=0
\end{array}
\right.
\end{eqnarray}
where $\hat{G}^{n-1}_k:=\hat{G}^{n-1}_k(t,x,\vec{\D}_t^k\Omega_{n-1},\D_x\vec{\D}_t^k\Omega_{n-1},\vec{\D}_t^k\alpha_{n-1},\D_x\vec{\D}_t^k\alpha_{n-1})$ and $\vec{\D}_t^k\omega:=(\omega,\D_t\omega,\cdots,\D_t^k\omega)$.
Defining
\begin{eqnarray*}
\A_n^k(t):=\frac{1}{J}\sum\limits_{i=0}^{k}\int\Big|\D_t^i\alpha_n\Big|^J
\end{eqnarray*}
and applying the same methods of \textbf{Step 1}, we have
\begin{eqnarray*}
\frac{d}{dt}\A_n^k\leq C\A_n^k+||\hat{G}_k^{n-1}||^J_{L^J}.
\end{eqnarray*}
Tedious computation gives
\begin{eqnarray*}
|\D_x^j\hat{G}^{n-1}_p|&\leq&\sum C_{abqr}(|\alpha_{n-1}|+|\Omega_{n-1}|,|\D_x\Omega_{n-1}|)\prod\limits_{l=1}^r|\D_x^{1+e_l}\D_t^{h_l}\Omega_{n-1}|^{\mu_l}\\
&&\prod\limits_{e=1}^q\left[\prod\limits_{u=1}^e|\D_x^{a_u}\D_t^{b_u}\alpha_{n-1}|^{\tau_u}\prod\limits_{o=1}^{q-e}\left(|\D_x^{f_o}\D_t^{g_o}\alpha_{n-1}|+|\D_x^{f_o}\D_t^{g_o}\Omega_{n-1}|\right)^{\nu_o}\right]\\
&&+\sum C_{\bar{a}\bar{b}\bar{q}\bar{r}}(|\Omega_{n-1}|,|\D_x\alpha_{n-1}|+|\D_x\Omega_{n-1}|)\prod\limits_{c=1}^{\bar{q}}|\D_x^{l_c}\D_t^{n_c}\Omega_{n-1}|^{\kappa_c}\\
&&\prod\limits_{v=1}^{\bar{r}}\left[\prod\limits_{w=1}^{\bar{r}-v}\left(|\D_x^{1+z_w}\D_t^{y_w}\alpha_{n-1}|+|\D_x^{1+z_w}\D_t^{y_w}\Omega_{n-1}|\right)^{\lambda_w}\prod\limits_{h=1}^v|\D_x^{1+x_h}\D_t^{i_h}\alpha_{n-1}|^{\bar{\mu}_h}\right]\\
&\leq&\sum\sum\limits_{i_1+i_2=1}^{\infty}2^{i_1}|C^{abqr}_{i_1i_2}|\.|\alpha_{n-1}|^{i_1}|\Omega_{n-1}|^{i_1}|\D_x\Omega_{n-1}|^{i_2}\prod\limits_{l=1}^r|\D_x^{1+e_l}\D_t^{h_l}\Omega_{n-1}|^{\mu_l}\\
&&\prod\limits_{e=1}^q\left[\prod\limits_{u=1}^e|\D_x^{a_u}\D_t^{b_u}\alpha_{n-1}|^{\tau_u}\prod\limits_{o=1}^{q-e}\left(2^{\nu_o-1}|\D_x^{f_o}\D_t^{g_o}\alpha_{n-1}|^{\nu_o}|\D_x^{f_o}\D_t^{g_o}\Omega_{n-1}|^{\nu_o}\right)\right]\\
&&+\sum\sum\limits_{\bar{i}_1+\bar{i}_2=1}^{\infty}C^{\bar{a}\bar{b}\bar{q}\bar{r}}_{\bar{i}_1\bar{i}_2}2^{\bar{i}_2}|\Omega_{n-1}|^{\bar{i}_1}|\D_x\alpha_{n-1}|^{\bar{i}_2}|\D_x\Omega_{n-1}|^{\bar{i}_2}\prod\limits_{c=1}^{\bar{q}}|\D_x^{l_c}\D_t^{n_c}\Omega_{n-1}|^{\kappa_c}\\
&&\prod\limits_{v=1}^{\bar{r}}\left[\prod\limits_{w=1}^{\bar{r}-v}\left(2^{\lambda_w}|\D_x^{1+z_w}\D_t^{y_w}\alpha_{n-1}|^{\lambda_w}|\D_x^{1+z_w}\D_t^{y_w}\Omega_{n-1}|^{\lambda_w}\right)\prod\limits_{h=1}^v|\D_x^{1+x_h}\D_t^{i_h}\alpha_{n-1}|^{\bar{\mu}_h}\right]
\end{eqnarray*}
with
$$
a+\sum\limits_{e=1}^q\left(\sum\limits_{o=1}^{q-e}g_o\nu_o+\sum\limits_{u=1}^eb_u\tau_u\right)+\sum\limits_{l=1}^rh_l\mu_l=p\s\mbox{and}\s b+\sum\limits_{l=1}^r(1+e_l)\mu_l+\sum\limits_{e=1}^q\left(\sum\limits_{o=1}^{q-e}f_o\nu_o+\sum\limits_{u=1}^ea_u\tau_u\right)=j
$$
and
$$
\bar{a}+\sum\limits_{c=1}^{\bar{q}}n_c\kappa_c+\sum\limits_{v=1}^{\bar{r}}\left(\sum\limits_{w=1}^{\bar{r}-v}y_w\lambda_w+\sum\limits_{h=1}^vi_h\bar{\mu}_h\right)=p\s\mbox{and}\s \bar{b}+\sum\limits_{c=1}^{\bar{q}}l_c\kappa_c+\sum\limits_{v=1}^{\bar{r}}\left[\sum\limits_{w=1}^{\bar{r}-v}(1+z_w)\lambda_w+\sum\limits_{h=1}^v(1+x_h)\bar{\mu}_h\right]=j.
$$
General Minkowski's Inequality and H\"{o}lder's Inequality lead to
\begin{eqnarray*}
&&||\D_x^j\hat{G}^{n-1}_p||_{L^J}\leq\sum\sum\limits_{i_1+i_2=1}^{\infty}2^{i_1}|C^{abqr}_{i_1i_2}|\.\Big|\Big|\,|\alpha_{n-1}|^{i_1}|\Omega_{n-1}|^{i_1}|\D_x\Omega_{n-1}|^{i_2}\prod\limits_{l=1}^r|\D_x^{1+e_l}\D_t^{h_l}\Omega_{n-1}|^{\mu_l}\\
&&\prod\limits_{e=1}^q\left[\prod\limits_{u=1}^e|\D_x^{a_u}\D_t^{b_u}\alpha_{n-1}|^{\tau_u}\prod\limits_{o=1}^{q-e}\left(2^{\nu_o}|\D_x^{f_o}\D_t^{g_o}\alpha_{n-1}|^{\nu_o}|\D_x^{f_o}\D_t^{g_o}\Omega_{n-1}|^{\nu_o}\right)\right]\Big|\Big|_{L^J}\\
&&+\sum\sum\limits_{\bar{i}_1+\bar{i}_2=1}^{\infty}C^{\bar{a}\bar{b}\bar{q}\bar{r}}_{\bar{i}_1\bar{i}_2}2^{\bar{i}_2}\.\Big|\Big|\,|\Omega_{n-1}|^{\bar{i}_1}|\D_x\alpha_{n-1}|^{\bar{i}_2}|\D_x\Omega_{n-1}|^{\bar{i}_2}\prod\limits_{c=1}^{\bar{q}}|\D_x^{l_c}\D_t^{n_c}\Omega_{n-1}|^{\kappa_c}\\
&&\prod\limits_{v=1}^{\bar{r}}\left[\prod\limits_{w=1}^{\bar{r}-v}\left(2^{\lambda_w}|\D_x^{1+z_w}\D_t^{y_w}\alpha_{n-1}|^{\lambda_w}|\D_x^{1+z_w}\D_t^{y_w}\Omega_{n-1}|^{\lambda_w}\right)\prod\limits_{h=1}^v|\D_x^{1+x_h}\D_t^{i_h}\alpha_{n-1}|^{\bar{\mu}_h}\right]\Big|\Big|_{L^J}\\
&&\leq\sum\sum\limits_{i_1+i_2=1}^{\infty}2^{i_1}|C^{abqr}_{i_1i_2}|\.||\alpha_{n-1}||_{L^{i_1\mathfrak{P}}}^{i_1}||\Omega_{n-1}||_{L^{i_1\mathfrak{P}}}^{i_1}||\D_x\Omega_{n-1}||_{L^{i_2\mathfrak{P}}}^{i_2}\prod\limits_{l=1}^r||\D_x^{1+e_l}\D_t^{h_l}\Omega_{n-1}||_{L^{\mu_l\mathfrak{P}}}^{\mu_l}\\
&&\prod\limits_{e=1}^q\left[\prod\limits_{u=1}^e||\D_x^{a_u}\D_t^{b_u}\alpha_{n-1}||_{L^{\tau_u\mathfrak{P}}}^{\tau_u}\prod\limits_{o=1}^{q-e}\left(2^{\nu_o}||\D_x^{f_o}\D_t^{g_o}\alpha_{n-1}||_{L^{\mathfrak{P}\nu_o}}^{\nu_o}||\D_x^{f_o}\D_t^{g_o}\Omega_{n-1}||_{L^{\mathfrak{P}\nu_o}}^{\nu_o}\right)\right]\\
&&+\sum\sum\limits_{\bar{i}_1+\bar{i}_2=1}^{\infty}2^{\bar{i}_2}|C^{\bar{a}\bar{b}\bar{q}\bar{r}}_{\bar{i}_1\bar{i}_2}|\.||\Omega_{n-1}||_{L^{\mathfrak{K}\bar{i}_1}}^{\bar{i}_1}||\D_x\alpha_{n-1}||_{L^{\mathfrak{K}\bar{i}_2}}^{\bar{i}_2}||\D_x\Omega_{n-1}||_{L^{\mathfrak{K}\bar{i}_2}}^{\bar{i}_2}\prod\limits_{c=1}^{\bar{q}}||\D_x^{l_c}\D_t^{n_c}\Omega_{n-1}||_{L^{\mathfrak{K}\kappa_c}}^{\kappa_c}\\
&&\prod\limits_{v=1}^{\bar{r}}\left[\prod\limits_{w=1}^{\bar{r}-v}\left(2^{\lambda_w}||\D_x^{1+z_w}\D_t^{y_w}\alpha_{n-1}||_{L^{\mathfrak{K}\lambda_w}}^{\lambda_w}||\D_x^{1+z_w}\D_t^{y_w}\Omega_{n-1}||_{L^{\mathfrak{K}\lambda_w}}^{\lambda_w}\right)\prod\limits_{h=1}^v||\D_x^{1+x_h}\D_t^{i_h}\alpha_{n-1}||_{L^{\mathfrak{K}\mu_h}}^{\bar{\mu}_h}\right]
\end{eqnarray*}
with
$$
\mathfrak{P}:=(3+r+3q^2/2-q/2)J\leq(3+\min\{j,p\}/2+3\min\{j,p\}^2/2)J\leq b(j+p)
$$
and
$$
\mathfrak{K}:=(3+\bar{q}+3\bar{r}^2/2-\bar{r}/2)J\leq(3+\min\{j,p\}/2+3\min\{j,p\}^2/2)J\leq b(j+p).
$$
Sobolev's Embedding implies
\begin{eqnarray}\label{jxw5}
&&||\D_x^j\hat{G}^{n-1}_p||_{L^J}\nonumber\\
&\leq&\sum\sum\limits_{i_1+i_2=1}^{\infty}2^{i_1}|C^{abqr}_{i_1i_2}|c(i_1\mathfrak{P})^{2i_1}c(i_2\mathfrak{P})^{i_2}\.||\alpha_{n-1}||_{W_J^{\lceil i_1\mathfrak{P}\rceil_J,0}}^{i_1}||\Omega_{n-1}||_{W_J^{\lceil i_1\mathfrak{P}\rceil_J,0}}^{i_1}||\Omega_{n-1}||_{W_J^{1+\lceil i_2\mathfrak{P}\rceil_J,0}}^{i_2}\nonumber\\
&&\prod\limits_{l=1}^rc(\mu_l\mathfrak{P})^{\mu_l}||\Omega_{n-1}||_{W_J^{1+e_l+\lceil\mu_l\mathfrak{P}\rceil_J,h_l}}^{\mu_l}\prod\limits_{e=1}^q\left[\prod\limits_{u=1}^ec(\tau_u\mathfrak{P})^{\tau_u}\prod\limits_{o=1}^{q-e}2^{\nu_o}c(\nu_o\mathfrak{P})^{2\nu_o}\right]\nonumber\\
&&\prod\limits_{e=1}^q\left[\prod\limits_{u=1}^e||\alpha_{n-1}||_{W_J^{a_u+\lceil\tau_u\mathfrak{P}\rceil_J,b_u}}^{\tau_u}\prod\limits_{o=1}^{q-e}\left(||\alpha_{n-1}||_{W_J^{f_o+\lceil\mathfrak{P}\nu_o\rceil_J,g_o}}^{\nu_o}||\Omega_{n-1}||_{W_J^{f_o+\lceil\mathfrak{P}\nu_o\rceil_J,g_o}}^{\nu_o}\right)\right]\nonumber\\
&&+\sum\sum\limits_{\bar{i}_1+\bar{i}_2=1}^{\infty}2^{\bar{i}_2}|C^{\bar{a}\bar{b}\bar{q}\bar{r}}_{\bar{i}_1\bar{i}_2}|c(\mathfrak{K}\bar{i}_1)^{\bar{i}_1}c(\mathfrak{K}\bar{i}_2)^{2\bar{i}_2}\.||\Omega_{n-1}||_{W_J^{\lceil\mathfrak{K}\bar{i}_1\rceil_J,0}}^{\bar{i}_1}||\alpha_{n-1}||_{W_J^{1+\lceil\mathfrak{K}\bar{i}_2\rceil_J,0}}^{\bar{i}_2}||\Omega_{n-1}||_{W_J^{1+\lceil\mathfrak{K}\bar{i}_2\rceil_J,0}}^{\bar{i}_2}\nonumber\\
&&\prod\limits_{c=1}^{\bar{q}}c(\kappa_c\mathfrak{K})^{\kappa_c}\prod\limits_{c=1}^{\bar{q}}||\Omega_{n-1}||_{W_J^{l_c+\lceil\mathfrak{K}\kappa_c\rceil_J,n_c}}^{\kappa_c}\prod\limits_{v=1}^{\bar{r}}\left[\prod\limits_{w=1}^{\bar{r}-v}2^{\lambda_w}c(\lambda_w\mathfrak{K})^{2\lambda_w}\prod\limits_{h=1}^{v}c(\bar{\mu}_h\mathfrak{K})^{\bar{\mu}_h}\right]\nonumber\\
&&\prod\limits_{v=1}^{\bar{r}}\left[\prod\limits_{w=1}^{\bar{r}-v}\left(||\alpha_{n-1}||_{W_J^{1+z_w+\lceil\mathfrak{K}\lambda_w\rceil_J,y_w}}^{\lambda_w}||\Omega_{n-1}||_{W_J^{1+z_w+\lceil\mathfrak{K}\lambda_w\rceil_J,y_w}}^{\lambda_w}\right)\prod\limits_{h=1}^v||\alpha_{n-1}||_{W_J^{1+x_h+\lceil\mathfrak{K}\bar{\mu}_h\rceil_J,i_h}}^{\bar{\mu}_h}\right]\nonumber\\
&\leq&\sum\sum\limits_{i_1+i_2=1}^{\infty}2^{i_1+i_2+2\min\{j,p\}}c((i_1+i_2)b(j,p))^{2(i_1+i_2)}|C^{abqr}_{i_1i_2}|c(\min\{j,p\}b(j,p))^{4\min\{j,p\}}\nonumber\\
&&||\alpha_{n-1}||_{W_J^{2,0}}^{i_1}||\Omega_{n-1}||_{W_J^{2,0}}^{i_1+i_2}||\Omega_{n-1}||_{W_J^{1+j,p}}^{2\min\{j,p\}}||\alpha_{n-1}||_{W_J^{j+1,p}}^{2\min\{j,p\}}\nonumber\\
&&+\sum\sum\limits_{\bar{i}_1+\bar{i}_2=1}^{\infty}2^{\bar{i}_1+\bar{i}_2+\min\{j,p\}}c((\bar{i}_1+\bar{i}_2)b(j,p))^{2(\bar{i}_1+\bar{i}_2)}|C^{\bar{a}\bar{b}\bar{q}\bar{r}}_{\bar{i}_1\bar{i}_2}|c(\min\{j,p\}b(j,p))^{4\min\{j,p\}}\nonumber\\
&&||\alpha_{n-1}||_{W_J^{2,0}}^{\bar{i}_2}||\Omega_{n-1}||_{W_J^{2,0}}^{\bar{i}_1+\bar{i}_2}||\Omega_{n-1}||_{W_J^{1+j,p}}^{2\min\{j,p\}}||\alpha_{n-1}||_{W_J^{1+j,p}}^{2\min\{j,p\}}\nonumber\\
&\leq&2^{p+j}c(\min\{j,p\}b(j,p))^{4\min\{j,p\}}\sum\limits_{i_1+i_2=1}^{\infty}2^{i_1+i_2+2\min\{j,p\}}c((i_1+i_2)b(j,p))^{2(i_1+i_2)}\sum\limits_{a+b+q+r=0}^{j+p}|C^{abqr}_{i_1i_2}|\nonumber\\
&&||\alpha_{n-1}||_{W_J^{2,0}}^{i_1+i_2}||\Omega_{n-1}||_{W_J^{2,0}}^{i_1+i_2}||\Omega_{n-1}||_{W_J^{1+j,p}}^{2\min\{j,p\}}||\alpha_{n-1}||_{W_J^{j+1,p}}^{2\min\{j,p\}}.
\end{eqnarray}
\textbf{Condition 1} tells us
\begin{eqnarray*}
\sum\limits_{a+b+q+r=0}^{j+p}|C^{abqr}_{i_1i_2}|\leq c((i_1+i_2)b(j+p+2))^{-2i_1-2i_2}(j+p+2)^{i_1+i_2}/(i_1+i_2)!
\end{eqnarray*}
Defining $\m^p_k(n):=||\alpha_{n}||_{W_J^{k,p}}$ and substituting the above inequality into $(\ref{jxw5})$, we obtain
\begin{eqnarray}\label{lwq}
||\D_x^j\hat{G}^{n-1}_p||_{L^J}&\leq& 8^{p+j}c((j+p)b(j+p))^{4(j+p)}\ex\left\{(2j+2p+4)||\alpha_{n-1}||_{W^{2,0}_J}||\Omega_{n-1}||_{W_J^{2,0}}\right\}\nonumber\\
&&\left(1+||\Omega_{n-1}||_{W_J^{1+j,p}}\right)^{2(j+p)}\left(1+||\alpha_{n-1}||_{W_J^{j+1,p}}\right)^{2(j+p)}\\
&=&8^{p+j}c((j+p)b(j+p))^{4(j+p)}\ex\left\{(2j+2p+4)\m_2^0(n-1)\N_2^0(n-1)\right\}\nonumber\\
&&\left[1+\N_{1+j}^p(n-1)\right]^{2(j+p)}\left[1+\m_{1+j}^p(n-1)\right]^{2(j+p)}\nonumber
\end{eqnarray}
and
\begin{eqnarray}\label{jbc}
||\hat{G}^{n-1}_p||_{L^J}&\leq& 8^{p}c(pb(p))^{4p}\ex\left\{(2p+4)||\alpha_{n-1}||_{W^{2,0}_J}||\Omega_{n-1}||_{W_J^{2,0}}\right\}\nonumber\\
&&\left(1+||\Omega_{n-1}||_{W_J^{1,p}}\right)^{2p}\left(1+||\alpha_{n-1}||_{W_J^{1,p}}\right)^{2p}\\
&=&8^{p}c(pb(p))^{4p}\ex\left\{(2p+4)\m_2^0(n-1)\N_2^0(n-1)\right\}\nonumber\\
&&\left[1+\N_{1}^p(n-1)\right]^{2p}\left[1+\m_{1}^p(n-1)\right]^{2p}.\nonumber
\end{eqnarray}
From $(\ref{jxw})$ it follows that
\begin{eqnarray*}
\N^p_1(n-1)&\leq& C(1,J)2^{2p-3}\left(1+\mathcal{E}^{p+1}_{n-1}\right)^{(p-2)/J}\c(J(1+p)^2)^{1+p}\ex\left\{(1+p)\N^0_{2}(n-1)\right\}\\
&:=&a(p,\mathcal{E}^{p+1}_{n-1},\N^0_{2}(n-1))
\end{eqnarray*}
Now we are going to estimate $\m_1^p(n-1)$. From Lemma \ref{lemma2.1} and $(\ref{eq:3'})$ it follows that
\begin{eqnarray}\label{11'}
&&||\D_t^p\alpha_{n-1}||_{W^{k,0}_J}\leq C(k,J)\left\{||\D_t^{p+1}\alpha_{n-1}||_{W^{k-1,0}_J}+(1+||\D_t^p\alpha_{n-1}||_{W^{k-1,0}_J})^{2k}\right\}\nonumber\\
&&+C(k,J)8^{p+k}c((k+p)b(k+p))^{4(k+p)}\ex\left\{(2k+2p)||\alpha_{n-1}||_{W^{2,0}_J}||\Omega_{n-1}||_{W_J^{2,0}}\right\}\\
&&\left(1+||\Omega_{n-1}||_{W_J^{k-1,p}}\right)^{2(k+p)}\left(1+||\alpha_{n-1}||_{W_J^{k-1,p}}\right)^{2(k+p)}.\nonumber
\end{eqnarray}
$(\ref{11'})$ leads to
\begin{eqnarray*}
&&\m^p_k(n-1)\leq C(k,J)\left\{\m^{p+1}_{k-1}(n-1)+(1+\m^p_{k-1}(n-1))^{2k}\right\}\\
&&+C(k,J)8^{p+k}c((k+p)b(k+p))^{4(k+p)}\ex\left\{(2k+2p)\m_2^0(n-1)\N_2^0(n-1)\right\}\\
&&\left(1+\N_{k-1}^p(n-1)\right)^{2(k+p)}\left(1+\m_{k-1}^p(n-1)\right)^{2(k+p)}.
\end{eqnarray*}
Specially we have
\begin{eqnarray*}
&&\m^p_1(n-1)\leq C(1,J)J^{1/J}(1+\A^{p+1}_{n-1})^{2/J}\\
&&+C(1,J)8^{p+1}c((1+p)b(1+p))^{4(1+p)}\ex\left\{(2+2p)\m_2^0(n-1)\N_2^0(n-1)\right\}\\
&&\left(1+\mathcal{E}^{p+1}_{n-1}\right)^{2(1+p)/J}\left(1+\A^{p+1}_{n-1}\right)^{2(1+p)/J}:=b(p,\A^{p+1}_{n-1},\mathcal{E}^{p+1}_{n-1},\m_2^0(n-1),\N_2^0(n-1)).
\end{eqnarray*}
Hence it is not difficult to get
\begin{eqnarray*}
\frac{d}{dt}\A_n^k\leq C\A_n^k+L(k,\N_2^0(n-1),\m_2^0(n-1),\mathcal{E}^{k+1}_{n-1},\A^{k+1}_{n-1}),
\end{eqnarray*}
where
\begin{eqnarray*}
&&L(k,\N_2^0(n-1),\m_2^0(n-1),\mathcal{E}^{k+1}_{n-1},\A^{k+1}_{n-1}):=8^{Jk}c(kb(k))^{4Jk}\ex\left\{J(2k+4)\m_2^0(n-1)\N_2^0(n-1)\right\}\\
&&\left[1+a(k,\mathcal{E}^{k+1}_{n-1},\N^0_{2}(n-1))\right]^{2Jk}\left[1+b(k,\A^{k+1}_{n-1},\mathcal{E}^{k+1}_{n-1},\m_2^0(n-1),\N_2^0(n-1))\right]^{2Jk},
\end{eqnarray*}
which implies
\begin{eqnarray*}
\A_n^k(t)\leq\int_0^tds\,\exp\{C(t-s)\}L(k,\N_2^0(n-1)(s),\m_2^0(n-1)(s),\mathcal{E}^{k+1}_{n-1}(s),\A^{k+1}_{n-1}(s)).
\end{eqnarray*}

Defining $\B_n^k(t):=\mathcal{E}_n^k(t)+\A_n^k(t)$
and 
$$
Q(k,\N_2^0(n-1),\m_2^0(n-1),\B^{k+1}_{n-1}):=L(k,\N_2^0(n-1),\m_2^0(n-1),\mathcal{B}^{k+1}_{n-1},\B^{k+1}_{n-1})+H(k,\mathcal{B}_{n-1}^{k+1},\N_2^0(n-1))
$$
leads to
\begin{eqnarray*}
\B_n^k(t)\leq\exp(Ct)\exp\{\mu(0,k)k!\}+\int_0^tds\,\exp\{C(t-s)\}Q(k,\N_2^0(n-1)(s),\m_2^0(n-1)(s),\B^{k+1}_{n-1}(s)),
\end{eqnarray*}
where we have used the fact that 
$$
\B_n^k(0)=\mathcal{E}_n^k(0)\leq\exp\{\mu(0,k)k!\}.
$$
Lemma $\ref{lemma2.1}$ and Sobolev's Interpolation tell us
\begin{eqnarray*}
\N^0_2(n-1)&\leq&C(2,J)\left\{1+||\D_t\Omega_{n-1}||_{L^J}+||F(t,\.,\Omega_{n-2},\D_x\Omega_{n-2})||_{L^J}+\e\N_2^0(n-1)+\e^{-1/2}||\Omega_{n-1}||_{L^J}\right\}\\
&\leq&C(2,J)\left\{1+||\D_t\Omega_{n-1}||_{L^J}+\ex[2\N_2^0(n-2)]+\e\N_2^0(n-1)+\e^{-1/2}||\Omega_{n-1}||_{L^J}\right\},
\end{eqnarray*}
where we have used $(\ref{ldy})$. Taking $\e:=2^{-1}C(2,J)^{-1}$ yields
\begin{eqnarray*}
\N^0_2(n-1)&\leq&2C(2,J)^2\left\{1+J^{1/J}\(\mathcal{E}^1_{n-1}\)^{1/J}+\ex[2\N_2^0(n-2)]\right\}\\
&\leq&2C(2,J)^2\left\{1+J^{1/J}\(\mathcal{B}^{1}_{n-1}\)^{1/J}+\ex[2\N_2^0(n-2)]\right\}.
\end{eqnarray*}
Similar method gives
\begin{eqnarray*}
\m^0_2(n-1)\leq2C(2,J)^2\left\{1+J^{1/J}\(\mathcal{B}^{1}_{n-1}\)^{1/J}+\ex[4\N_2^0(n-2)\m^0_2(n-2)]\right\}.
\end{eqnarray*}
Defining $\W_2^0(n):=\N_2^0(n)+\m^0_2(n)$, we have
\begin{eqnarray}\label{zy1}
\W^0_2(n)&\leq&4C(2,J)^2\left\{1+J^{1/J}\(\mathcal{B}^{1}_{n}\)^{1/J}+\ex[4\W_2^0(n-1)^2]\right\}
\end{eqnarray}
which implies
\begin{eqnarray}\label{drrl}
8\W^0_2(n)^2&\leq&2^9C(2,J)^4\left\{1+J^{2/J}\(\mathcal{B}^{1}_{n}\)^{2/J}+\ex[8\W_2^0(n-1)^2]/L\right\}
\end{eqnarray}
	Note that $\ex(x)\leq x$ provided $x\leq\log L$. We hope 
\begin{eqnarray}\label{lsnf}
8\W_2^0(n)^2\leq\log L\s\s\mbox{for all $n\in\mathbb{N}$}
\end{eqnarray}
and
\begin{eqnarray}\label{plznf}
8\W^0_2(n)^2&\leq&2^9C(2,J)^4\left\{1+J^{2/J}\(\mathcal{B}^{1}_{n}\)^{2/J}+8\W_2^0(n-1)^2/L\right\}\nonumber\\
&\leq&2^9C(2,J)^4\left\{1+J^{2/J}\(\mathcal{B}^{1}_{n}\)^{2/J}+\log L/L\right\}\leq\log L.
\end{eqnarray}

Letting $R(k,\W_2^0(n-1),\B^{k+1}_{n-1}):=Q(k,\W_2^0(n-1),\W_2^0(n-1),\B^{k+1}_{n-1})$, one are able to get
\begin{eqnarray}\label{zy2}
\B_n^k(t)\leq\exp\{Ct+\mu(0,k)k!\}+\exp(Ct)\int_0^tds\,R(k,\W_2^0(n-1)(s),\B^{k+1}_{n-1}(s)).
\end{eqnarray}
From the definition it follows that $\W_2^0(0)=||\omega_0||_{W_J^{2,0}}+||\Omega_1-\omega_0||_{W_J^{2,0}}$. Lemma $\ref{lemma5.2}$ tells us
\begin{eqnarray*}
&&\W_2^0(0)\leq\mathfrak{L}(t):=||\omega_0||_{W^{2,0}_J}\\
&&+e^{Ct}\int_0^tds\,\left\{||\L_A[F(s,\.,\omega_0,\D_x\omega_0)]-\L_A^2\omega_0||^J_{L^J}+||\D_{\vec{n}}[F(s,\.,\omega_0,\D_x\omega_0)]-\D_{\vec{n}}\L_A\omega_0||^J_{L^J(\p M)}\right\}
\end{eqnarray*}
and Lemma $\ref{lemma5.3}$ gives
\begin{eqnarray*}
&&\B_0^k(t)\leq||\omega_0||^J_{L^J}/J+Ce^{Ct}\left[\sum\limits_{r+q=0}^{k-1}||\L_A^q[\D_t^rF(0,\.,\omega_0,\D_x\omega_0)]||_{L^J}^J+\sum\limits_{i=0}^k||\L_A^i\omega_0||_{L^J}^J\right]\\
&&+Ce^{Ct}\int_0^t\left[\sum\limits_{i=0}^k||\D_t^iF(s,\.,\omega_0,\D_x\omega_0)||^J_{L^J}+\sum\limits_{i=0}^k||\D_t^{i}\theta(s,\.)||_{L^J(\p M)}^J\right]ds:=\mathfrak{B}^k(t).
\end{eqnarray*}
Define $\hat{\B}_n^k(t):=\sup\limits_{0\leq s\leq t}\B_n^k(s)$. As soon as $(\ref{lsnf})$ and $(\ref{plznf})$ hold true, we will obtain
\begin{eqnarray*}
\hat{\B}_n^k(t)\leq\exp\{Ct+\mu(0,k)k!\}+t\exp(Ct)R\left(k,[\log L/8]^{1/2},\hat{\B}^{k+1}_{n-1}(t)\right).
\end{eqnarray*}
From the definition it follows that
\begin{eqnarray*}
R(k,[\log L/8]^{1/2},\hat{\B}^{k+1}_{n-1}(t))\leq\Psi(L,k)\left[1+||\theta(t)||^J	_{W^{1,0}_J}+\hat{\B}^{k+1}_{n-1}(t)\right]^{18k^3}+||\D^k_t\theta(t)||^J_{L^J(\p M)}
\end{eqnarray*}
where 
$$
\Psi(L,k):=2^{27Jk^3}k^{8Jk^2}L^{3Jk^3}C(2,J)^{4Jk^2}c((1+k)b(1+k))^{8J(1+k)^3}J^{26Jk^3},
$$
which implies
\begin{eqnarray}\label{zy3}
1+||\theta(t)||^J	_{W^{1,0}_J}+\hat{\B}_n^k(t)&\leq&1+||\theta(t)||^J	_{W^{1,0}_J}+||\D^k_t\theta(t)||^J_{L^J(\p M)}+\exp\{Ct+\mu(0,k)k!\}\nonumber\\
&&+t\exp(Ct)\Psi(L,k)\left[1+||\theta(t)||^J	_{W^{1,0}_J}+\hat{\B}^{k+1}_{n-1}(t)\right]^{18k^3}\\
&\leq&1+||\theta(t)||^J	_{W^{1,0}_J}+||\D^k_t\theta(t)||^J_{L^J(\p M)}+\exp\{Ct+\mu(0,k)k!\}\nonumber\\
&&+(1+t)e^{Ct}\Psi(L,k)\left[1+||\theta(t)||^J	_{W^{1,0}_J}+\hat{\B}^{k+1}_{n-1}(t)\right]^{18k^3}\nonumber
\end{eqnarray}
Hence we get, for $k\geq1$,
\begin{eqnarray*}
[(k-1)!]^3\log\(1+||\theta(t)||^J	_{W^{1,0}_J}+\hat{\B}_n^k(t)\)\leq[(k-1)!]^3\tau(L,k,t)+18(k!)^3\log\left[1+||\theta(t)||^J	_{W^{1,0}_J}+\hat{\B}^{k+1}_{n-1}(t)\right],
\end{eqnarray*}
where 
\begin{eqnarray*}
\tau(L,k,t):=\log\left\{3+2||\theta(t)||^J_{W^{1,0}_J}+||\D_t^k\theta(t)||^J_{L^J}\right\}+\log\left[\Psi(L,k)(1+t)\right]+2Ct+\mu(0,k)k!.
\end{eqnarray*}
We rewrite the above estimation as
\begin{eqnarray*}
&&[(k+n-i-1)!]^3\log\(1+||\theta(t)||^J	_{W^{1,0}_J}+\hat{\B}^{k+n-i}_i(t)\)\\
&\leq&[(k+n-i-1)!]^3\tau(L,k+n-i,t)+18[(k+n-i)!]^3\log\left[1+||\theta(t)||^J	_{W^{1,0}_J}+\hat{\B}^{k+n-i+1}_{i-1}(t)\right]
\end{eqnarray*}
for $1\leq i\leq n$, which means
\begin{eqnarray}\label{lz}
&&[(k-1)!]^3\log\(1+||\theta(t)||^J	_{W^{1,0}_J}+\hat{\B}^{k}_n(t)\)\\
&\leq&\sum\limits_{i=1}^n18^{n-i}[(k+n-i-1)!]^3\tau(L,k+n-i,t)+18^n[(k+n-1)!]^3\log\left[1+||\theta(t)||^J	_{W^{1,0}_J}+\hat{\B}^{k+n}_{0}(t)\right]\nonumber.
\end{eqnarray}
From the above it follows that
\begin{eqnarray*}
&&\log\(1+||\theta(t)||^J	_{W^{1,0}_J}+\hat{\B}^{k}_n(t)\)\leq\sum\limits_{i=0}^{n-1}18^{i}\left[i!\binom{k-1+i}{k-1}\right]^3\tau(L,k+i,t)\\
&&+18^n\left[n!\binom{k-1+n}{k-1}\right]^3\log\left[1+||\theta(t)||^J	_{W^{1,0}_J}+\hat{\B}^{k+n}_{0}(t)\right].
\end{eqnarray*}
$(\ref{zy3})$ tells us
\begin{eqnarray*}
&&\log\left[1+||\theta(t)||^J	_{W^{1,0}_J}+\hat{\B}_{n+1}^k(t)\right]\leq\log t-\log(1+t)+\tau(L,k,t)+18k^3\log\left[1+||\theta(t)||^J	_{W^{1,0}_J}+\hat{\B}^{k+1}_{n}(t)\right]\\
&&\leq\sum\limits_{i=0}^{n-1}18^{i+1}\left[i!\binom{k+i}{k}k\right]^3\tau(L,k+1+i,t)+18^{n+1}\left[n!\binom{k+n}{k}k\right]^3\log\left[1+||\theta(t)||^J	_{W^{1,0}_J}+\hat{\B}^{k+1+n}_{0}(t)\right]\\
&&+\log t+\tau(L,k,t)\\
&&\leq\sum\limits_{i=0}^{n-1}18^{i+1}\left[i!\binom{k+i}{k}k\right]^3\tau(L,k+1+i,t)+18^{n+1}\left[n!\binom{k+n}{k}k\right]^3\log\left[1+||\theta(t)||^J	_{W^{1,0}_J}+\mathfrak{B}^{k+1+n}(t)\right]\\
&&+\log t+\tau(L,k,t).
\end{eqnarray*}
Specially we have, for $n\geq2$,
\begin{eqnarray}\label{kz}
&&\log\left[1+||\theta(t)||^J	_{W^{1,0}_J}+\hat{\B}_{n}^1(t)\right]\\
&&\leq\sum\limits_{i=0}^{n-2}18^{i+1}\left[(i+1)!\right]^3\tau(L,2+i,t)+18^{n}\left(n!\right)^3\log\left[1+||\theta(t)||^J	_{W^{1,0}_J}+\mathfrak{B}^{1+n}(t)\right]\nonumber\\
&&+\log t+\tau(L,1,t)\nonumber
\end{eqnarray}
and
\begin{eqnarray}\label{mz}
\log\(1+||\theta(t)||^J	_{W^{1,0}_J}+\hat{\B}^{1}_1(t)\)\leq\tau(L,1,t)+18\log\left[1+||\theta(t)||^J	_{W^{1,0}_J}+\mathfrak{B}^{2}(t)\right].
\end{eqnarray}
In the next, we shall determine $L$. Recalling $(\ref{lsnf})$ and $(\ref{plznf})$ we hope
$$
8\mathfrak{L}(t)^2\leq\log L
$$
and
\begin{eqnarray*}
&&2^9C(2,J)^4\left\{1+J^{2/J}\(\mathcal{B}^{1}_{n}\)^{2/J}+\log L/L\right\}\\
&\leq&J2^9C(2,J)^4\left(1+\hat{\B}^{1}_{n}\right)^{2/J}\leq J2^9C(2,J)^4\left(1+||\theta(t)||^J	_{W^{1,0}_J}+\hat{\B}^{1}_{n}\right)^{2/J}\leq\log L.
\end{eqnarray*}
Combining $(\ref{kz})$ and $(\ref{mz})$ we need
\begin{eqnarray}\label{yh}
&&\log J+9\log2+4\log C(2,J)+\frac{2}{J}\log\left[1+||\theta(t)||^J	_{W^{1,0}_J}+\hat{\B}_{n}^1(t)\right]\nonumber\\
&&\leq\frac{2}{J}\sum\limits_{i=0}^{n-2}18^{i+1}\left[(i+1)!\right]^3\tau(L,2+i,t)+\frac{2}{J}18^{n}\left(n!\right)^3\log\left[1+||\theta(t)||^J	_{W^{1,0}_J}+\mathfrak{B}^{1+n}(t)\right]\nonumber\\
&&+\frac{2}{J}\log t+\frac{2}{J}\tau(L,1,t)+\log J+9\log2+4\log C(2,J)\nonumber\\
&&\leq\log\log L\s\s\mbox{for $n\geq2$}
\end{eqnarray}
and
\begin{eqnarray}\label{yy}
&&\log J+9\log2+4\log C(2,J)+\frac{2}{J}\log\left[1+||\theta(t)||^J	_{W^{1,0}_J}+\hat{\B}_{1}^1(t)\right]\nonumber\\
&&\leq\log J+9\log2+4\log C(2,J)+\frac{2}{J}\tau(L,1,t)+\frac{36}{J}\log\left[1+||\theta(t)||^J	_{W^{1,0}_J}+\mathfrak{B}^{2}(t)\right]\nonumber\\
&&\leq\log\log L
\end{eqnarray}
and
\begin{eqnarray}\label{zz}
&&\log J+9\log2+4\log C(2,J)+\frac{2}{J}\log\left[1+||\theta(t)||^J	_{W^{1,0}_J}+\mathfrak{B}^1(t)\right]\nonumber\\
&&\leq\log\log L.
\end{eqnarray}
To ensure $(\ref{yh})$, $(\ref{yy})$ and $(\ref{zz})$ hold ture, we only have to make sure that
\begin{eqnarray*}
&&\sum\limits_{i=0}^{n-2}18^{i+1}\left[(i+1)!\right]^3\tau(L,2+i,t)+18^{n}\left(n!\right)^3\log\left[1+||\theta(t)||^J	_{W^{1,0}_J}+\mathfrak{B}^{1+n}(t)\right]\\
&&\leq18^n(n!)^3\left[\tau(L,n,t)+||\theta(t)||^J	_{W^{1,0}_J}+\mathfrak{B}^{1+n}(t)\right]\\
&&\leq18^n(n!)^3\sup\limits_{t\in[0,T)}\left[\tau(n,n,t)+||\theta(t)||^J	_{W^{1,0}_J}+\mathfrak{B}^{1+n}(t)\right]\\
&&\leq\tilde{C}<\infty
\end{eqnarray*}
where $\tilde{C}$ is independent of $n$. Recalling $(\ref{lz})$ gives
\begin{eqnarray*}
&&[(k-1)!]^3\log\(1+||\theta(t)||^J	_{W^{1,0}_J}+\hat{\B}^{k}_n(t)\)\\
&\leq&18^n[(k+n-1)!]^3\left[\tau(L,k+n-1,t)+||\theta(t)||^J	_{W^{1,0}_J}+\mathfrak{B}^{k+n}(t)\right]\\
&\leq&18^{k+n}[(k+n)!]^3\left[\tau(L,k+n,t)+||\theta(t)||^J	_{W^{1,0}_J}+\mathfrak{B}^{k+n}(t)\right]\\
&\leq&18^{k+n}[(k+n)!]^3\left[\tau(k+n,k+n,t)+||\theta(t)||^J	_{W^{1,0}_J}+\mathfrak{B}^{k+n}(t)\right]\leq\tilde{C}\s\s\mbox{for $t\in[0,T)$},
\end{eqnarray*}
which is equivalent to
\begin{eqnarray}\label{sml}
\hat{\B}^{k}_n(t)\leq\exp\{[(k-1)!]^{-3}\tilde{C}\}-||\theta(t)||^J	_{W^{1,0}_J}-1\s\s\mbox{for $t\in[0,T)$}.
\end{eqnarray}
Specially, we can get
\begin{eqnarray*}
\mathcal{E}^{1}_n(t)\leq e^{\tilde{C}}-||\theta(t)||^J	_{W^{1,0}_J}-1\s\s\mbox{for $t\in[0,T)$},
\end{eqnarray*}
which implies
$$
||\Omega_n(t)||^J_{L^J(M)}+||\L_A\Omega_n(t)||^J_{L^J(M)}\leq\tilde{E}(||\theta(t)||^J	_{W^{1,0}_J},\tilde{C}).
$$
From $(\ref{sml})$ it follows that we have to ensure
\begin{eqnarray*}
||\theta(t)||^J	_{W^{1,0}_J}<\exp\{[(k-1)!]^{-3}\tilde{C}\}-1\s\s\mbox{for $t\in[0,T)$}.
\end{eqnarray*}
From Banach's Fixed Point Theorem it follows that there is a unique solution $\omega\in L^{\infty}([0,T),W^{2,k}_J)$ to $(\ref{eq:2})$ for small $T>0$.

At last, we use $(\ref{eq:2})$ and Lemma \ref{lemma2.1} to transform the time derivatives to spatial derivatives. That is to say, $\omega$ lies in $\bigcap\limits_{r+2l\leq 2k+2}L^{\infty}([0,T),W^{r,l}_J)$. 
\end{proof}
\section{Appendix}
\begin{lemma}\label{lml}
If $x,y\geq\log2$, then we have $\ex(x+y)\geq \ex(x)+\ex(y)$.
\end{lemma}
\begin{proof}
The proof is elementary and we leave it to the readers.
\end{proof}

\begin{lemma}\label{lemma2.1}
Suppose $A$ is smooth and positive-definite on $[0,T)\times M$. Then we have
$$
||\omega||_{W^{k+2,0}_p}\leq C(k,p)\left\{||\L_A\omega||_{W^{k,0}_p}+||\D_{\vec{n}}\omega||_{W^{k+1,0}_p}+(1+||\omega||_{W^{k+1,0}_p})^{2k+1}\right\}
$$
for $p\in(1,\infty)$, where the constant $C(k,p)$ also depends on $A$ and is non-decreasing with respect to $k$.
\end{lemma}
\begin{proof}
It is obvious to get
$$
||A(\slashed{\Delta}\omega)||_{W^{k,0}_p}\leq||\L_A\omega||_{W^{k,0}_p}+||g^{-1}\sharp(\D_xA)(\D_x\omega)||_{W^{k,0}_p}.
$$
Induction argument gives
$$
|\D_x^j\slashed{\Delta}\omega|\lesssim|A(\D_x^j\slashed{\Delta}\omega)|\lesssim|\D_x^j[A(\slashed{\Delta}\omega)]|+\sum\prod\limits_{i_j=1}^{s_j}|\D_x^{p_{i_j}}\omega|^{q_{i_j}}|\D_x^{r_{i_j}}\slashed{\Delta}\omega|\s\mbox{with}\s j=\sum\limits_{i_j=1}^{s_j}(p_{i_j}q_{i_j}+r_{i_j}),\s p_{i_j}q_{i_j}\geq 1
$$
and
$$
|\D_x^j[g^{-1}\sharp(\D_xA)(\D_x\omega)]|\lesssim\sum\prod\limits_{\tilde{i}_j=1}^{\tilde{s}_j}|\D_x^{\tilde{p}_{\tilde{i}_j}}\omega|^{\tilde{q}_{\tilde{i}_j}}|\D_x^{\tilde{r}_{\tilde{i}_j}+1}\omega|\s\mbox{with}\s j=\sum\limits_{\tilde{i}_j=1}^{\tilde{s}_j}(\tilde{p}_{\tilde{i}_j}\tilde{q}_{\tilde{i}_j}+\tilde{r}_{\tilde{i}_j}),\s\tilde{p}_{\tilde{i}_j}\geq1.
$$
H\"{o}lder's Inequality leads to
$$
||\D_x^j\slashed{\Delta}\omega||_{L^p}\lesssim||\D_x^j[A(\slashed{\Delta}\omega)]||_{L^p}+\sum\prod\limits_{i_j=1}^{s_j}||\D_x^{p_{i_j}}\omega||_{L^{q_{i_j}\tau_{i_j}}}^{q_{i_j}}||\D_x^{r_{i_j}}\slashed{\Delta}\omega||_{L^{\mu_{i_j}}}
$$
and
$$
||\D_x^j[g^{-1}\sharp(\D_xA)(\D_x\omega)]||_{L^p}\lesssim\sum\prod\limits_{\tilde{i}_j=1}^{\tilde{s}_j}||\D_x^{\tilde{p}_{\tilde{i}_j}}\omega||_{L^{\tilde{\tau}_{\tilde{i}_j}\tilde{q}_{\tilde{i}_j}}}^{\tilde{q}_{i_j}}||\D_x^{\tilde{r}_{\tilde{i}_j}+1}\omega||_{L^{\tilde{\mu}_{\tilde{i}_j}}}
$$
with $\mu_{i_j}, \tau_{i_j}, \tilde{\mu}_{\tilde{i}_j}, \tilde{\tau}_{\tilde{i}_j}$ to be determined and satisfing $\sum\limits_{i_j=1}^{s_j}(1/\mu_{i_j}+1/\tau_{i_j})=1/p$ and $\sum\limits_{\tilde{i}_j=1}^{\tilde{s}_j}(1/\tilde{\mu}_{\tilde{i}_j}+1/\tilde{\tau}_{\tilde{i}_j})=1/p$.
Combining the above conclusions we arrive at
\begin{eqnarray}\label{eq:16}
&&||\slashed{\Delta}\omega||_{W^{k,0}_p}\lesssim||\L_A\omega||_{W^{k,0}_p}+\sum\limits_{j=1}^k\sum\limits_{\substack{\sum\limits_{i_j=1}^{s_j}(p_{i_j}q_{i_j}+r_{i_j})=j\\ p_{i_j}q_{i_j}\geq 1}}\prod\limits_{i_j=1}^{s_j}||\D_x^{p_{i_j}}\omega||_{L^{q_{i_j}\tau_{i_j}}}^{q_{i_j}}||\D_x^{r_{i_j}+2}\omega||_{L^{\mu_{i_j}}}\\
&&+\sum\limits_{j=0}^k\sum\limits_{\substack{\sum\limits_{\tilde{i}_j=1}^{\tilde{s}_j}(\tilde{p}_{\tilde{i}_j}\tilde{q}_{\tilde{i}_j}+\tilde{r}_{\tilde{i}_j})=j\\ \tilde{p}_{\tilde{i}_j}\geq 1}}\prod\limits_{\tilde{i}_j=1}^{\tilde{s}_j}||\D_x^{\tilde{p}_{\tilde{i}_j}}\omega||_{L^{\tilde{\tau}_{\tilde{i}_j}\tilde{q}_{\tilde{i}_j}}}^{\tilde{q}_{\tilde{i}_j}}||\D_x^{\tilde{r}_{\tilde{i}_j}+1}\omega||_{L^{\tilde{\mu}_{\tilde{i}_j}}}\nonumber.
\end{eqnarray}
From Sobolev's Embedding it follows that
\begin{eqnarray}\label{eq:15}
||\D_x^{p_{i_j}}\omega||_{L^{\tau_{i_j}q_{i_j}}}^{q_{i_j}}\lesssim||\omega||^{q_{i_j}}_{W_p^{p_{i_j}+\lceil q_{i_j}\tau_{i_j}\rceil_p,0}}\s\s\mbox{and}\s\s||\D_x^{\tilde{p}_{\tilde{i}_j}}\omega||_{L^{\tilde{\tau}_{\tilde{i}_j}\tilde{q}_{\tilde{i}_j}}}^{\tilde{q}_{\tilde{i}_j}}\lesssim||\omega||^{\tilde{q}_{\tilde{i}_j}}_{W_p^{\tilde{p}_{\tilde{i}_j}+\lceil\tilde{q}_{\tilde{i}_j}\tilde{\tau}_{\tilde{i}_j}\rceil_p,0}}
\end{eqnarray}
\begin{eqnarray}\label{eq:14}
||\D_x^{r_{i_j}+2}\omega||_{L^{\mu_{i_j}}}\lesssim||\omega||_{W_p^{r_{i_j}+2+\lceil \mu_{i_j}\rceil_p,0}}\s\s\mbox{and}\s\s||\D_x^{\tilde{r}_{\tilde{i}_j}+1}\omega||_{L^{\tilde{\mu}_{\tilde{i}_j}}}\lesssim||\omega||_{W_p^{\tilde{r}_{\tilde{i}_j}+1+\lceil\tilde{\mu}_{\tilde{i}_j}\rceil_p,0}}.
\end{eqnarray}

Let us focus on the case when $j=k$, $i_j=s_j=1$, $\tilde{i}_j=\tilde{s}_j=1$, $r_{i_j}=k-1$, $\tilde{r}_{\tilde{i}_j}=k$, $p_{i_j}=1$, $\tilde{p}_{\tilde{i}_j}=0$, $\tilde{q}_{\tilde{i}_j}=0$ and $q_{i_j}=1$. We need $\lceil \mu_{i_j}\rceil_p=0$ which is equivalent to $1/\mu_{i_j}=1/p$ and $\lceil \tilde{\mu}_{\tilde{i}_j}\rceil_p=0$ which is equivalent to $1/\tilde{\mu}_{\tilde{i}_j}=1/p$.

As for the other cases, we need
$$p_{i_j}+\lceil q_{i_j}\tau_{i_j}\rceil_p\leq k+1,\s\tilde{p}_{\tilde{i}_j}+\lceil\tilde{q}_{\tilde{i}_j}\tilde{\tau}_{\tilde{i}_j}\rceil_p\leq k+1,\s r_{i_j}+\lceil \mu_{i_j}\rceil_p\leq k-1,\s \tilde{r}_{\tilde{i}_j}+\lceil\tilde{\mu}_{\tilde{i}_j}\rceil_p\leq k.$$
In order to ensure the above inequalities, we only need
$$
q_{i_j}/p-q_{i_j}(k-p_{i_j}+1)/m\leq1/\tau_{i_j},\s\s\tilde{q}_{\tilde{i}_j}/p-\tilde{q}_{\tilde{i}_j}(k-\tilde{p}_{\tilde{i}_j}+1)/m\leq 1/\tilde{\tau}_{\tilde{i}_j},
$$
$$
1/p-(k-r_{i_j}-1)/m\leq 1/\mu_{i_j},\s\s1/p-(k-\tilde{r}_{\tilde{i}_j})/m\leq 1/\tilde{\mu}_{\tilde{i}_j}.
$$
Simple computation gives
$$
\sum\limits_{i_j=1}^{s_j}\left[1/p-(k-r_{i_j}-1)/m+q_{i_j}/p-q_{i_j}(k-p_{i_j}+1)/m\right]=[1/p-(k+1)/m](s_j+\sum\limits_{i_j=1}^{s_j}q_{i_j})+(j+2)s_j/m
$$
and
$$
\sum\limits_{\tilde{i}_j=1}^{\tilde{s}_j}\left[1/p-(k-\tilde{r}_{\tilde{i}_j})/m+\tilde{q}_{\tilde{i}_j}/p-\tilde{q}_{\tilde{i}_j}(k-\tilde{p}_{\tilde{i}_j}+1)/m\right]=[1/p-(k+1)/m](\tilde{s}_j+\sum\limits_{\tilde{i}_j=1}^{\tilde{s}_j}\tilde{q}_{\tilde{i}_j})+(j+1)\tilde{s}_j/m.
$$
We hope
\begin{eqnarray}\label{eq:18}
[1/p-(k+1)/m](s_j+\sum\limits_{i_j=1}^{s_j}q_{i_j})+(j+2)s_j/m\leq 1/p
\end{eqnarray}
and
\begin{eqnarray}\label{eq:17}
[1/p-(k+1)/m](\tilde{s}_j+\sum\limits_{\tilde{i}_j=1}^{\tilde{s}_j}\tilde{q}_{\tilde{i}_j})+(j+1)\tilde{s}_j/m\leq 1/p.
\end{eqnarray}
However,
\begin{eqnarray*}
&&[1/p-(k+1)/m](s_j+\sum\limits_{i_j=1}^{s_j}q_{i_j})+(j+2)s_j/m-1/p\\
&\leq&[1/p-(k+1)/m](s_j+\sum\limits_{i_j=1}^{s_j}q_{i_j})+(k+2)s_j/m-1/p\\
&=&\sum\limits_{i_j=1}^{s_j}\{q_{i_j}[1/p-(k+1)/m]+1/m+1/p\}-1/p\\
&\leq&\sum\limits_{i_j=1}^{s_j}[1/p-(k+1)/m+1/m+1/p]-1/p\\
&=&(2/p-k/m)s_j-1/p\leq0
\end{eqnarray*}
where we have used the fact $k\geq 2m/p$. Similar trick leads to the desired inequality $(\ref{eq:17})$.

Substituting $(\ref{eq:15})$ and $(\ref{eq:14})$ into $(\ref{eq:16})$ yields
\begin{eqnarray*}
&&||\slashed{\Delta}\omega||_{W^{k,0}_p}\lesssim||\L_A\omega||_{W^{k,0}_p}+\sum\limits_{j=1}^k\sum\limits_{\substack{\sum\limits_{i_j=1}^{s_j}(p_{i_j}q_{i_j}+r_{i_j})=j\\ p_{i_j}q_{i_j}\geq 1}}||\omega||_{W^{k+1,0}_p}^{\sum\limits_{i_j=1}^{s_j}q_{i_j}+s_j}+\sum\limits_{j=0}^k\sum\limits_{\substack{\sum\limits_{\tilde{i}_j=1}^{\tilde{s}_j}(\tilde{p}_{\tilde{i}_j}\tilde{q}_{\tilde{i}_j}+\tilde{r}_{\tilde{i}_j})=j\\ \tilde{p}_{\tilde{i}_j}\geq1}}||\omega||_{W_p^{k+1,0}}^{\sum\limits_{\tilde{i}_j=1}^{\tilde{s}_j}\tilde{q}_{\tilde{i}_j}+\tilde{s}_j}\nonumber\\
&\leq&||\L_A\omega||_{W^{k,0}_p}+\sum\limits_{j=1}^k\sum\limits_{\substack{\sum\limits_{i_j=1}^{s_j}(p_{i_j}q_{i_j}+r_{i_j})=j\\ p_{i_j}q_{i_j}\geq 1}}(1+||\omega||_{W^{k+1,0}_p})^{\sum\limits_{i_j=1}^{s_j}q_{i_j}+s_j}+\sum\limits_{j=0}^k\sum\limits_{\substack{\sum\limits_{\tilde{i}_j=1}^{\tilde{s}_j}(\tilde{p}_{\tilde{i}_j}\tilde{q}_{\tilde{i}_j}+\tilde{r}_{\tilde{i}_j})=j\\ \tilde{p}_{\tilde{i}_j}\geq1}}(1+||\omega||_{W_p^{k+1,0}})^{\sum\limits_{\tilde{i}_j=1}^{\tilde{s}_j}\tilde{q}_{\tilde{i}_j}+\tilde{s}_j}\nonumber\\
&:=&||\L_A\omega||_{W^{k,0}_p}+I_{10}+I_{11}.
\end{eqnarray*}
For $I_{10}$, $\sum\limits_{i_j=1}^{s_j}(p_{i_j}q_{i_j}+r_{i_j})=j$ and $p_{i_j}q_{i_j}\geq 1$ implies $s_j\leq j$ and $\sum\limits_{i_j=1}^{s_j}q_{i_j}\leq j$. Hence,
$$
I_{10}\leq\sum\limits_{j=1}^kj(1+||\omega||_{W^{k+1,0}_p})^{2j}\leq k^2(1+||\omega||_{W^{k+1,0}_p})^{2k}.
$$
Similarly, $\sum\limits_{\tilde{i}_j=1}^{\tilde{s}_j}(\tilde{p}_{\tilde{i}_j}\tilde{q}_{\tilde{i}_j}+\tilde{r}_{\tilde{i}_j})=j$ and $\tilde{p}_{\tilde{i}_j}\geq1$ implies $\tilde{s}_j\leq j+1$ and $\sum\limits_{\tilde{i}_j=1}^{\tilde{s}_j}\tilde{q}_{\tilde{i}_j}\leq j$. Hence,
$$
I_{11}\leq\sum\limits_{j=0}^k(j+1)(1+||\omega||_{W^{k+1,0}_p})^{2j+1}\leq(k+1)^2(1+||\omega||_{W^{k+1,0}_p})^{2k+1}.
$$
In summary, we obtain
$$
||\slashed{\Delta}\omega||_{W^{k,0}_p}\lesssim||\L_A\omega||_{W^{k,0}_p}+(1+||\omega||_{W^{k+1,0}_p})^{2k+1}.
$$
From Theorem 3.2 of \cite{W} it follows that
$$
||\omega||_{W^{k+2,0}_p}\lesssim||\L_A\omega||_{W^{k,0}_p}+||\D_{\vec{n}}\omega||_{W^{k+1,0}_p}+(1+||\omega||_{W^{k+1,0}_p})^{2k+1}.
$$
\end{proof}
\begin{lemma}\label{lemma5.2}
Suppose $\omega$ is smooth and satisfies the following equation
\begin{eqnarray*}
\left\{
\begin{array}{rl}
&\D_t\omega+\L_A\omega=f\\
&\omega(0,\.)=\omega_0,\s\s\D_{\vec{n}}\omega=\theta,
\end{array}
\right.
\end{eqnarray*}
then, for $p\in[2,\infty)$, we have
\begin{eqnarray*}
||\L_A\omega||^p_{L^p}\leq e^{Ct}||\L_A\omega_0||_{L^p}^{p}+\int_0^tds\,e^{C(t-s)}\left(||\L_Af||_{L^p}^{p}+||\D_{\vec{n}}f-\D_t\theta||_{L^p(\p M)}^{p}\right).
\end{eqnarray*}
\end{lemma}
\begin{proof}
Tediou computation gives
\begin{eqnarray*}
\frac{d}{dt}||\L_A\omega||^p_{L^p}&=&p\int|\L_A\omega|^{p-2}\langle\L_A\omega,\L_Af\rangle-p\int_{\p M}|\L_A\omega|^{p-2}\langle\L_A\omega,A(\D_{\vec{n}}f-\D_t\theta)\rangle\\
&&-p\int|\L_A\omega|^{p-2}g^{-1}\sharp\langle\D_{x}\L_A\omega,A(\D_{x}\L_A\omega)\rangle\\
&&-p(p-2)\int|\L_A\omega|^{p-4}g^{-1}\sharp\langle\L_A\omega,A(\D_{x}\L_A\omega)\rangle\otimes\langle\L_A\omega,\D_{x}\L_A\omega\rangle\\
&\lesssim&\int|\L_A\omega|^{p-1}|\L_Af|+\int_{\p M}|\L_A\omega|^{p-1}|\D_{\vec{n}}f-\D_t\theta|-\int|\L_A\omega|^{p-2}|\D_{x}\L_A\omega|^2\\
&\lesssim&||\L_A\omega||_{L^p}^{p}+||\L_Af||_{L^p}^{p}+||\L_A\omega||_{L^p(\p M)}^{p}+||\D_{\vec{n}}f-\D_t\theta||_{L^p(\p M)}^{p}-\int|\L_A\omega|^{p-2}|\D_{x}\L_A\omega|^2.
\end{eqnarray*}
Boundary Trace Theorem tells us
\begin{eqnarray*}
||\L_A\omega||_{L^p(\p M)}^{p}&=&||\,|\L_A\omega|^{p}||_{L^1(\p M)}\lesssim||\,|\L_A\omega|^{p}||_{W^{1,0}_1}\\
&=&||\L_A\omega||_{L^p}^{p}+||\D_x(|\L_A\omega|^{p})||_{L^1}\leq||\L_A\omega||_{L^p}^{p}+p||\,|\L_A\omega|^{p-1}|\D_x\L_A\omega|\,||_{L^1}\\
&\lesssim&C(\epsilon)||\L_A\omega||_{L^p}^{p}+\epsilon||\,|\L_A\omega|^{p-2}|\D_x\L_A\omega|^2||_{L^1},
\end{eqnarray*}
where we have used Kato's Inequality. In conclusion, one are able to obtain
\begin{eqnarray*}
\frac{d}{dt}||\L_A\omega||^p_{L^p}\lesssim||\L_A\omega||_{L^p}^{p}+||\L_Af||_{L^p}^{p}+||\D_{\vec{n}}f-\D_t\theta||_{L^p(\p M)}^{p},
\end{eqnarray*}
which implies the desired result.
\end{proof}
\begin{lemma}\label{lemma5.3}
Suppose $\omega$ is smooth and satisfies the following equation
\begin{eqnarray}\label{cb}
\left\{
\begin{array}{rl}
&\D_t\omega+\L_A\omega=f\\
&\omega(0,\.)=\omega_0,\s\s\D_{\vec{n}}\omega=\theta,
\end{array}
\right.
\end{eqnarray}
then, for $p\in[2,\infty)$, we have
\begin{eqnarray*}
\A^k(t)\leq Ce^{Ct}\left[\sum\limits_{r+q=0}^{k-1}||\D_t^r\L_A^qf(0,\.)||_{L^p}^p+\sum\limits_{i=0}^k||\L_A^i\omega_0||_{L^p}^p\right]+C\int_0^te^{C(t-s)}\left[||f(s,\.)||^p_{W_p^{0,k}}+\sum\limits_{i=0}^k||\D_t^{i}\theta(s,\.)||_{L^p(\p M)}^p\right]ds,
\end{eqnarray*}
where $\A^k(t):=\sum\limits_{i=0}^k\int|\D_t^i\omega(t,\.)|^p/p$.
\end{lemma}
\begin{proof}
Taking $i$-th derivatives with respect to $t$ on both sides of $(\ref{cb})$ leads to
\begin{eqnarray*}
\left\{
\begin{array}{rl}
&\D_t^{i+1}\omega+\L_A\D_t^{i}\omega=\D_t^{i}f\\
&\omega(0,\.)=\omega_0,\s\s\D_t^{j}\omega(0,\.)=\D_t^{j-1}f-\L_A\D_t^{j-1}\omega(0,\.),\s\s\D_{\vec{n}}\D_t^{i}\omega=\D_t^{i}\theta
\end{array}
\right.
\end{eqnarray*}
for $1\leq j\leq i$. Tedious computation yields
\begin{eqnarray*}
\frac{d}{dt}\A^k&=&\sum\limits_{i=0}^k\int|\D_t^i\omega|^{p-2}\langle\D_t^i\omega,\D_t^{i}f\rangle-\sum\limits_{i=0}^k\int_{\p M}|\D_t^i\omega|^{p-2}\langle\D_t^i\omega,A(\D_t^{i}\theta)\rangle\\
&&-\sum\limits_{i=0}^k\int|\D_t^i\omega|^{p-2}\langle\D_x\D_t^i\omega,A(\D_x\D_t^{i}\omega)\rangle\sharp g^{-1}\\
&&-(p-2)\sum\limits_{i=0}^k\int|\D_t^i\omega|^{p-4}\langle\D_t^i\omega,A(\D_x\D_t^{i}\omega)\rangle\otimes\langle\D_t^i\omega,\D_x\D_t^{i}\omega\rangle\sharp g^{-1}\\
&\lesssim&\A^k+||f||^p_{W_p^{0,k}}+\sum\limits_{i=0}^k\int_{\p M}|\D_t^i\omega|^{p}+\sum\limits_{i=0}^k\int_{\p M}|\D_t^{i}\theta|^p-\sum\limits_{i=0}^k\int|\D_t^i\omega|^{p-2}|\D_x\D_t^i\omega|^2.
\end{eqnarray*}
Boundary Trace Theorem tells us
\begin{eqnarray*}
||\D_t^{i}\omega||_{L^p(\p M)}^{p}&=&||\,|\D_t^{i}\omega|^{p}||_{L^1(\p M)}\lesssim||\,|\D_t^{i}\omega|^{p}||_{W^{1,0}_1}\\
&=&||\D_t^{i}\omega||_{L^p}^{p}+||\D_x(|\D_t^{i}\omega|^{p})||_{L^1}\leq||\D_t^{i}\omega||_{L^p}^{p}+p||\,|\D_t^{i}\omega|^{p-1}|\D_x\D_t^{i}\omega|\,||_{L^1}\\
&\lesssim&C(\epsilon)||\D_t^{i}\omega||_{L^p}^{p}+\epsilon||\,|\D_t^{i}\omega|^{p-2}|\D_x\D_t^{i}\omega|^2||_{L^1},
\end{eqnarray*}
which implies
\begin{eqnarray*}
\frac{d}{dt}\A^k&\lesssim&\A^k+||f||^p_{W_p^{0,k}}+\sum\limits_{i=0}^k||\D_t^{i}\theta||_{L^p(\p M)}^p.
\end{eqnarray*}
Gronwall's Inequality leads to
\begin{eqnarray*}
\A^k(t)\leq e^{Ct}\A^k(0)+C\int_0^te^{C(t-s)}\left[||f(s,\.)||^p_{W_p^{0,k}}+\sum\limits_{i=0}^k||\D_t^{i}\theta(s,\.)||_{L^p(\p M)}^p\right]ds.
\end{eqnarray*}
In the next, we shall estimate $\A^k(0)$. Induction argument gives
$$
\D_t^i\omega(0,\.)=\sum\limits_{r+q=0}^{i-1}\D_t^r\circ(-\L_A)^qf(0,\.)+(-\L_A)^i\omega_0,
$$
which means
$$
\A^k(0)\lesssim\sum\limits_{r+q=0}^{k-1}||\D_t^r\L_A^qf(0,\.)||_{L^p}^p+\sum\limits_{i=0}^k||\L_A^i\omega_0||_{L^p}^p.
$$
\end{proof}

\textbf{Acknowledgement} The author is supported by Fundamental Research Funds for the Central Universities (No. 2021MS045).

\begin{tabular}{@{}r@{}p{16cm}@{}}
&Zonglin Jia, {\small{Department of Mathematics and Physics, North China Electric Power University,
\ Beijing, \ China}

\ Email: 50902525@ncepu.edu.cn}.
\end{tabular}

\end{document}